\documentclass[a4paper,11pt]{amsart}

\usepackage{amssymb}
\usepackage{amsfonts}
\usepackage{indentfirst}
\usepackage{graphicx, graphics}
\usepackage{pst-plot}

\newtheorem{theorem}{Theorem}[section]
\newtheorem{proposition}[theorem]{Proposition}

\newtheorem{remark}[theorem]{Remark}
\newtheorem{lemma}[theorem]{Lemma}
\newtheorem{definition}[theorem]{Definition}
\newtheorem{corollary}[theorem]{Corollary}

\newcommand\ep{{\varepsilon}}

\newcommand\zu{[0,1]}

\newcommand{\R}{\mathbb{R}}
\newcommand{\N}{\mathbb{N}}
\newcommand{\Z}{\mathbb{Z}}

\newcommand\om{\omega}

\newcommand{\calt}{\mathcal{T}}

\newcommand{\ho}{H\"older }

\newcommand{\mk}{\medskip}
\newcommand{\sk}{\smallskip}

\newcommand{\mcsib}{M^{*}_{I,\beta}}
\newcommand{\mcst}{M^{*}_{t}}
\newcommand{\OOO}{\Omega}

\newcommand{\aaa}{\alpha}

\newcommand{\bbb}{\beta}

\newcommand{\cac}{{\mathcal C}}
\newcommand{\cae}{{\mathcal E}}
\newcommand{\caf}{{\mathcal F}}

\newcommand{\cai}{{\mathcal I}}

\newcommand{\cat}{{\mathcal T}}

\newcommand{\ddd}{\delta}

\newcommand{\ds}{\displaystyle}
\newcommand{\eee}{\varepsilon}

\newcommand{\ess}{\emptyset}
\newcommand{\fff}{{\varphi}}

\newcommand{\hhh}{\eta}

\newcommand{\intt}{\mathrm{int}}
\newcommand{\kkk}{\kappa}
\newcommand{\mmm}{\mu}

\newcommand{\ooo}{\omega}
\newcommand{\oo}{\infty}

\newcommand{\sm}{\setminus}
\newcommand{\sse}{\subset}

\newcommand{\xo}{x_{0}}
\newcommand{\xe}{x_{1}}

\renewcommand{\ggg}{\gamma}

\renewcommand{\lll}{\lambda}
\newcommand{\hI}{\widehat{I}}
\newcommand{\tx}{\widetilde{x}}
\newcommand{\tr}{\widetilde{r}}
\newcommand{\txn}{\widetilde{x}_{n}}
\newcommand{\trn}{\widetilde{r}_{n}}
\begin{document}
\title[Measures and functions with prescribed  spectrum]{Measures and functions with prescribed homogeneous  multifractal  spectrum}

\author{Zolt\'an Buczolich
  \and St\'ephane Seuret }         

\thanks{
Research supported by the Hungarian
National Foundation for Scientific research K075242 and K104178.\\ }

% AMS subject classifications (used in AMS journals)
\subjclass[2000]{Primary 26A16, 28A80; Secondary 26A15, 28A78.}

   % AMS keywords (used in AMS journals)
   \keywords{ Multifractal  analysis, H\"older exponent. local dimensions, Hausdorff measures and dimension}

\address{Zolt\'an Buczolich,  E\"otv\"os University, 
Department of Analysis, P\'azm\'any P\'eter s\'et\'any 1/c, 1117 Budapest,
Hungary  -   St\'ephane Seuret,   LAMA, CNRS UMR 8050, Universit\'e Paris-Est Cr\'eteil,  
  61 avenue du G\'en\'eral de Gaulle, 
94 010 CR\'ETEIL Cedex, France}
\date{\today}

\begin{abstract}
In this paper we construct measures supported in   $\zu$  with prescribed  multifractal  spectrum. Moreover, these measures   are  homogeneously multifractal (HM, for short), in the sense that their restriction on any   subinterval of $\zu$ has the same  multifractal  spectrum as the whole measure. The spectra $f$ that we are able to prescribe  are  suprema  of a countable set of step functions supported by subintervals of $\zu$ and satisfy $f(h)\leq h$ for all $h\in [0,1]$. We also find a surprising constraint on the   multifractal  spectrum of a HM measure: the support of its spectrum within $[0,1]$ must be an interval. This result is a sort of Darboux theorem for multifractal spectra of measures. This result is optimal, since  we construct a HM measure with spectrum supported  by $\zu\cup \{ 2 \}$. Using wavelet theory, we also build HM functions with  prescribed multifractal spectrum.
\end{abstract}

\maketitle

%%%%%%%%%%%%%%%%%%%%%%%%%%%%%%%%%%
%%%%%%%%%%%%%%%%%%%%%%%%%%%%%%%%%%
%%%%%%%%%%%%%%%%%%%%%%%%%%%%%%%%%%
%%%%%%%%%%%%%%%%%%%%%%%%%%%%%%%%%%
%%%%%%%%%%%%%%%%%%%%%%%%%%%%%%%%%%
%%%%%%%%%%%%%%%%%%%%%%%%%%%%%%%%%%
%%%%%%%%%%%%%%%%%%%%%%%%%%%%%%%%%% 
\section{Introduction}

The  multifractal spectrum is now a widely spread issue in analysis.  It allows one to describe the local behavior of a given Borel measure, or a function. 
Our goal is to investigate the possible forms that a multifractal spectrum can take. We  obtain a   new Darboux-like theorem for the  spectrum of homogeneous multifractal measures, and we are able to construct measures   with  prescribed  non-homogeneous and homogeneous multifractal spectrum obtained as suprema of countable sets of step functions, when the local dimensions of the measure are less than 1. Using wavelet methods, we extend our result to non-monotone functions.

\mk

  Before exposing our results, let us introduce how the local regularity   is quantified. Recall that the support of a Borel positive measure, denoted by Supp($\mu$),  is the smallest closed set $E\in \R^d$ such that $\mu(\R^d\setminus E)=0$. 
 %%%%%%%%%%%%%%%%%%%%%%%%%
\begin{definition}
The local regularity of a positive Borel measure $\mu$ on $\R^d$ at a given $x_0\in$ Supp$(\mu)$ is quantified  by the (lower)  local dimension $h_\mu(x_0)$  (also called  local \ho exponent), defined as  
\begin{equation}
\label{defexpmu}
h_\mu(x_0)=\liminf_{r\to 0^+}  \frac{\log \mu (B(x_0,r) )} { \log r},
\end{equation}
where $B(x_0,r)$ denotes the open ball with center $x_0$ and radius $r$.  When $x_0\notin$ Supp$(\mu)$, by convention we set $h_\mu(x_0)=+\infty$.

Let $Z \in L^\infty_{loc}(\R^d)$, and $\alpha>0$.  The function   $Z$ is said to belong to the space $C^\alpha_{x_0}$ if there are a polynomial $P$ of degree less than $[\alpha]$ and a constant $C>0$ such that
\begin{equation}
\label{defpoint}
\mbox{for every $x$ in a neighborhood of $x_0$, } \ |Z(x) - P(x-x_0)| \leq C |x-x_0|^\alpha.
\end{equation}
The pointwise \ho exponent of $Z$ at $x_0$ is $h_Z(x_0) = \sup\{\alpha\geq 0: \ f\in C^\alpha_{x_0} \}.$ \end{definition}
%%%%%%%%%%%%%%%%%%%%%%%%%

\smallskip

We will mainly focus on real functions and measures, i.e. $d=1$. Even more, we will consider only measures and functions whose support are included in $\zu$: this has no influence on our dimensions questions.

\smallskip

Observe that when $h_Z(x_0) \le 1$, the pointwise \ho exponent of $Z$ at $x_0$ is also given by the formula
\begin{equation}
\label{defexpZ}
h_Z(x_0)= \liminf_{x \to x_0} \frac{\log |Z(x)-Z(x_0)|}{\log |x-x _0|}.
\end{equation}

Of course, there is a correspondence between the exponents and the spectrum of a measure $\mu$ and the same quantities for its integral $\displaystyle F_\mu(x)= \mu([0,x])$. Comparing formulae  \eqref{defexpmu} and \eqref{defexpZ}, we see that when  $h_\mu(x_0)< 1$,  $h_{F_\mu}(x_0)=h_\mu(x_0)$.  Problems may occur when  $h_\mu(x_0) \geq  1$ due to the presence of a polynomial in \eqref{defpoint}. Since we  mainly  consider  exponents less than 1  for measures, the problem may rise only when    $h_\mu(x)=1$. In order to guarantee that $h_\mu(x)=1$ is equivalent to $h_{F_\mu}(x) =1$ in our results, we add to each measure  another fixed measure ${\widetilde\mu}$ satisfying $h_{F_{\widetilde\mu} }(x) =1$ for all $x$  (see Proposition \ref{homone} for the existence of ${\widetilde\mu}$). Hence, in the sequel, {\bf we   work  equivalently with continuous monotone functions on the interval $\zu$, or with diffuse measures supported on  $\zu$}.

\begin{definition}
The  multifractal  spectrum of a  measure $\mu$ ({\em resp.} a  function $Z$) is the mapping $d_\mu$  ({\em resp.} $d_Z$) defined as
$$
 h\geq 0 \ \longmapsto \ d_\mu(h) := \dim E_\mu(h),   \ \ \ \ \   \ \ \ \mbox{ (\,{\em resp.}  \ $d_Z(h) := \dim E_Z(h)$)}
 $$
 where 
\begin{equation}
\label{defeh}
   E_\mu(h) :=\{x: h_\mu(x) =h\},  \ \ \ \mbox{ (\,{\em resp.}  \ $  E_Z(h) :=\{x: h_Z(x) =h\}$)}. 
\end{equation}
 By convention, we set $\dim \emptyset = -\infty$.  We call the support of the multifractal spectrum of $\mu$  ({\em resp.}    of $Z$) the set
\begin{equation}
\label{defsupport}
   \mbox{Support\,($d_\mu$)}=\{h\geq 0:d_\mu(h)\geq 0\} \ \ \mbox{(\,{\em resp.}     Support\,($d_Z$)=}\{h\geq 0:d_Z(h)\geq 0\} ). 
\end{equation}

 \end{definition}
Observe that  $d_\mu(h) \geq 0$
 as soon as there is a point in $E_\mu(h)$. The multifractal spectrum describes the geometrical distribution of the singularities of the measure or the function under consideration.
It is natural to investigate the   forms possibly taken by a  multifractal  spectrum, this is our goal in this paper.

\mk

A first question concerns  the prescription of the local dimension of a measure or a function. The pointwise \ho exponents of functions are well understood  \cite{JAFFPRES,DLVM,SJLV}: given a continuous function $f:\zu\to\R$, the mapping $x \mapsto h_f(x)$ is the liminf of a sequence of positive continuous functions, and, reciprocally, any liminf of a sequence of positive continuous functions is  indeed the map $x \mapsto h_f(x)$ associated with some continuous function~$f$. 

It is much more delicate to deal with the local dimensions of measures, as stated by the following lemma.
\begin{lemma}
\label{lem0}
Let $\mu$ be a  probability measure supported
on the closure of 
 an open set $\Omega\subset \R^d$. If  the mapping   $x\mapsto h_\mu(x)$ is continuous,   then it is     constant equal to $d$.
\end{lemma}
Lemma \ref{lem0} is quite easy (see Section \ref{sec_prel}), but it helps to understand that, from the local regularity standpoint,  measures are less flexible than functions.

\begin{definition}
A measure $\mu $  supported in $\zu$  is    homogeneously multifractal (for short, HM) when the restriction of $\mu$ to  any non-degenerate subinterval of $\zu$  has the same multifractal spectrum, i.e. for any non-empty subinterval $U\subset \zu$, 
$$\mbox{for any $h\geq 0$, } \  \dim \{x\in  U: h_\mu(x)=h\} = 
\dim \{x\in \zu: h_\mu(x)=h\} = d_\mu(h).$$
A similar definition applies to HM functions $Z$.
\end{definition}

Some results related to the prescription of multifractal spectrum have already been obtained.
The main result   is due to  Jaffard in \cite{J2}, and concerns only general functions (non-monotone). Jaffard proved, by wavelet methods, that any  supremum of a countable family of positive  piecewise constant functions with support in $(0,+\infty)$ is the  multifractal  spectrum of a function $Z:\zu\to\R$, which is  non-HM and non-monotone.  This  interesting result  yields no information on the possible spectra of measures (which are integrals of positive non-decreasing functions), nor on HM measures or functions.

\medskip

In this article, we  obtain the first results  for the prescription of  multifractal spectrum of measures, both HM and non-HM, and also of  HM functions.  In addition, when the measures $\mu$ are homogeneously multifractal,   we discover  a surprising constraint on $  \mbox{Support\,($d_\mu$)}$. 
 
\mk
 
Our first result is a Darboux-like  theorem for multifractal spectra of  HM  measures (thus, it holds also for monotone functions) for exponents less than 1. Darboux's theorem for a real differentiable function $Z:\R\to \R$ asserts that the image $Z'(I)$ of any interval $I\subset \R$ by the derivative of $Z$ is an interval. We obtain a similar result  for  HM  measures with exponents less than one: the support of the  multifractal  spectrum of $\mu$,  $ \mbox{Support\,($d_\mu$)}$ defined by \eqref{defsupport},  always contains an interval.  In other words, there is {\bf no spectrum gap for exponents less than 1 in the multifractal spectrum of HM measures}.  The necessary connectedness of  $  \mbox{Support\,($d_\mu$)}$ when $\mu $ is HM is a delicate issue. Establishing conditions under which   $ \mbox{Support\,($d_\mu$)}$ computed using limit exponents (not liminf exponents, as we do here) is necessarily connected, would be very interesting and useful in some situations, for instance for self-similar measures with overlaps, self-affine measures, or Bernoulli convolutions, see \cite{Feng} for instance.  We prove the following.
 
 %%%%%%%%%%%%%%%%%%%%%%%%%%%%%%
\begin{theorem}
\label{thdarboux}
For any non-atomic HM measure $\mu$ supported on $\zu$,  $ \mbox{Support\,($d_\mu$)} \,\cap \, \zu$ is necessarily an interval of  the form  $[\alpha,1] $, where $0\leq \alpha \leq 1$.
 \end{theorem}
%%%%%%%%%%%%%%%%%%%%%%%%%%%%%%

 Theorem \ref {thdarboux} is false for  non-monotone   functions:   the non-differentiable Riemann function $\displaystyle \sum_{n\geq 1} \frac{\sin \pi n^2 x}{n^2}$ is  HM, but its  multifractal  spectrum is supported by the set $[1/2,3/4]\cup\{3/2\}$, which admits an isolated point \cite{JAFFRIEMANN}. This theorem is also false for exponents greater than one, as stated  by the following complementary result.
 
%%%%%%%%%%%%%%%%%%%%%%%%%%%%%%
\begin{proposition}\label{*spectrgap}
There  is  a HM
measure $\mmm$ on $[0,1]$  such that  $ \mbox{Support\,($d_\mu$)}=[0,1]\cup \{ 2 \}$.
\end{proposition}
%%%%%%%%%%%%%%%%%%%%%%%%%%%%%%
 Nevertheless, one can derive   a theorem equivalent to Theorem \ref{thdarboux} for the limsup exponents greater than 1 (i.e. when the liminf in \eqref{defexpmu} is replaced by a limsup) using an argument of inversion of measures, i.e. by considering the inverse measure $\mu^{-1}$ of $\mu$ defined as $\mu^{-1}([0,\mu([0,x] ])= x$.  This inversion procedure transforms liminf exponents for $\mu$ into limsup exponents for $\mu^{-1}$.  The argument is due to Mandelbrot and Riedi in \cite{MR},  and the details are left to the reader. The associated result using the upper multifractal spectrum defined as $\displaystyle \overline{d_\mu}(h)=\dim\left\{x: \limsup_{r\to 0^+} \frac{\log \mu(B(x,r))}{\log r}=h\right\}$ is:
 %%%%%%%%%%%%%%%%%%%%%%%%%%%%%%
\begin{theorem}
\label{thdarboux2}
For any non-atomic HM measure supported on $\zu$, Support($ \overline{d_\mu}) \, \cap \, [1,+
\infty)$ is necessarily an interval of  the form $[1,\alpha ]$  where $\alpha \in [1,+\infty]$.
 \end{theorem}
%%%%%%%%%%%%%%%%%%%%%%%%%%%%%%

Our second main result  deals with the prescription of multifractal spectrum of a non-HM  measure. It is known (see Proposition \ref{prop2}) that the  multifractal  spectrum of a probability measure $\mu$ always satisfies $d_\mu(h)\leq \min(h,1)$, for every $h\geq 0$. 

\sk
 
Let us introduce the functions which are the candidates to be a multifractal spectrum. 

\begin{definition}
For every function $f: \zu\to\zu\cup\{-\infty\}$,  we define  Support($f$) = $\{x: f(x) > -\oo\}$ and   Support$^*$($f$) as the  smallest interval of the form $[h,1]$ containing  Support($f$). 
\end{definition}
 This definition is analogous to the definition of the support of the multifractal spectrum of a measure or a function as defined in equation \eqref{defsupport}.
 
%%%%%%%%%%%%%%%%%%%%%%%%%%%%%%
\begin{definition}
\label{defF}
The set   $\mathcal{F}$ consists of functions $f:[0,1]\to [0,1]\cup \{ -\oo \}$ satisfying the following:
For each $f \in \mathcal{F}$, there exists a countable family of functions $(f_n)_{n\geq 1}$,
$f_{n}:[0,+\oo]\to [0,1]\cup \{ -\oo \}$ such that

\begin{itemize}
\item
for every $n\geq 1$, Support($f_n$) is a closed, possibly degenerate interval $I_n \subset \zu$, 
\item
 $\inf_{n\geq 1}  \min (I_n) >0$,

\item
$f_n$ is constant over $I_n$,

\item
 for every $x\in I_n$, $f_n(x)\leq x$,
 
\item
for every $x\in \bigcup_{n\geq 1} I_n$, $f (x) = \sup_{n \geq 1} \,f_n(x)$. 
\end{itemize}
 \end{definition}
%%%%%%%%%%%%%%%%%%%%%%%%%%%%%%

The set $\mathcal{F}$ contains for instance the continuous functions and  the lower semi-continuous functions (provided that they satisfy $f(x)\leq x$) supported by subintervals of $\zu$, and one can also allow functions $f_{n}$ with degenerate, one-point supports.
 We prove the following: 

 %%%%%%%%%%%%%%%%%%%%%%%%%%%%%%
\begin{theorem}
\label{th111}
For every $f\in \caf$, there is a   Borel probability measure $\mu$,  supported by $\zu$ such that the multifractal spectrum of $\mu$   satisfies:
 \begin{enumerate}
\sk
\item
for all $ h\in { Support (f)}\setminus\{1\}$, $   d_\mmm(h) =   f(h) , $
\sk\item
 $d_{\mmm}(h) = -\infty $  if $h\in \R^+\sm$Support$(f)$,
\sk \item
 the set of points $\{x \in \zu:   h_\mmm(x)  =1\}$ has Lebesgue measure 1.
 %\sk\item
  %the set   $\{x \in \zu:   h_\mu(x) >1  \}$ is empty.
   \end{enumerate}
   In particular, $ \mbox{Support\,($d_\mu$)}=$Support($f)\cup\{1\}$.
 \end{theorem}
%%%%%%%%%%%%%%%%%%%%%%%%%%%%%%

The proof of Theorem \ref{th111} is somewhat classical: we concatenate measures $\mu_n$ whose spectra are (close to be) the functions $f_n$ (used in Definition  \ref{defF}). When dealing with HM measures, the result is different:

 %%%%%%%%%%%%%%%%%%%%%%%%%%%%%%
\begin{theorem}
\label{th11}
For every $f\in \caf$, there is a HM Borel probability measure $\mu$,  supported by $\zu$ such that the multifractal spectrum of $\mu$   satisfies:
 \begin{enumerate}
\sk
\item
for all $ h\in \mbox {Support}^*(f)\setminus\{1\}$, $   d_\mmm(h) =  \max(f(h),0), $
\sk\item
 $d_{\mmm}(h) = -\infty $  if $h\in \R^+\sm$Support$^*(f)$,
\sk \item
 the set of points $\{x \in \zu:   h_\mmm(x)  =1\}$ has Lebesgue measure 1.
 %\sk\item
  %the set   $\{x \in \zu:   h_\mu(x) >1  \}$ is empty.
   \end{enumerate}
   In particular, $ \mbox{Support\,($d_\mu$)}=$Support$^*$($f) $.
 \end{theorem}
%%%%%%%%%%%%%%%%%%%%%%%%%%%%%%

The notation $ d_\mu(h) = \max(f(h),0)$ indicates that either $h\in Support(f)$,  which implies  $f(h)\geq 0$ and $d_\mu(h)=f(h)$, or $h\in$ Support$^*(f)\setminus$ Support($f$),  and in this second case $d_\mu(h)=0$ (except  for $h=1$, for which $d_\mu(1)=1$) (see Figure \ref{fig11}).

Observe that, recalling Theorem \ref{thdarboux}, in Theorem \ref{th11} the 
intersection of $ \mbox{Support\,($d_\mu$)}$  with $\zu$  must be an interval. This is why we introduced Support$^*$($f$), and why Theorems \ref{th111} and \ref{th11} differ.   Theorem \ref{th11} is optimal for the class of functions $\mathcal{F}$ when $ \mbox{Support\,($d_\mu$)}$ is included in $\zu$.

%%%%%%%%%%%%%%%%%%%%%%%%%%%%%%%%%% 
\begin{center}
\begin{figure}
  \includegraphics[width=5.5cm,height = 4.0cm]{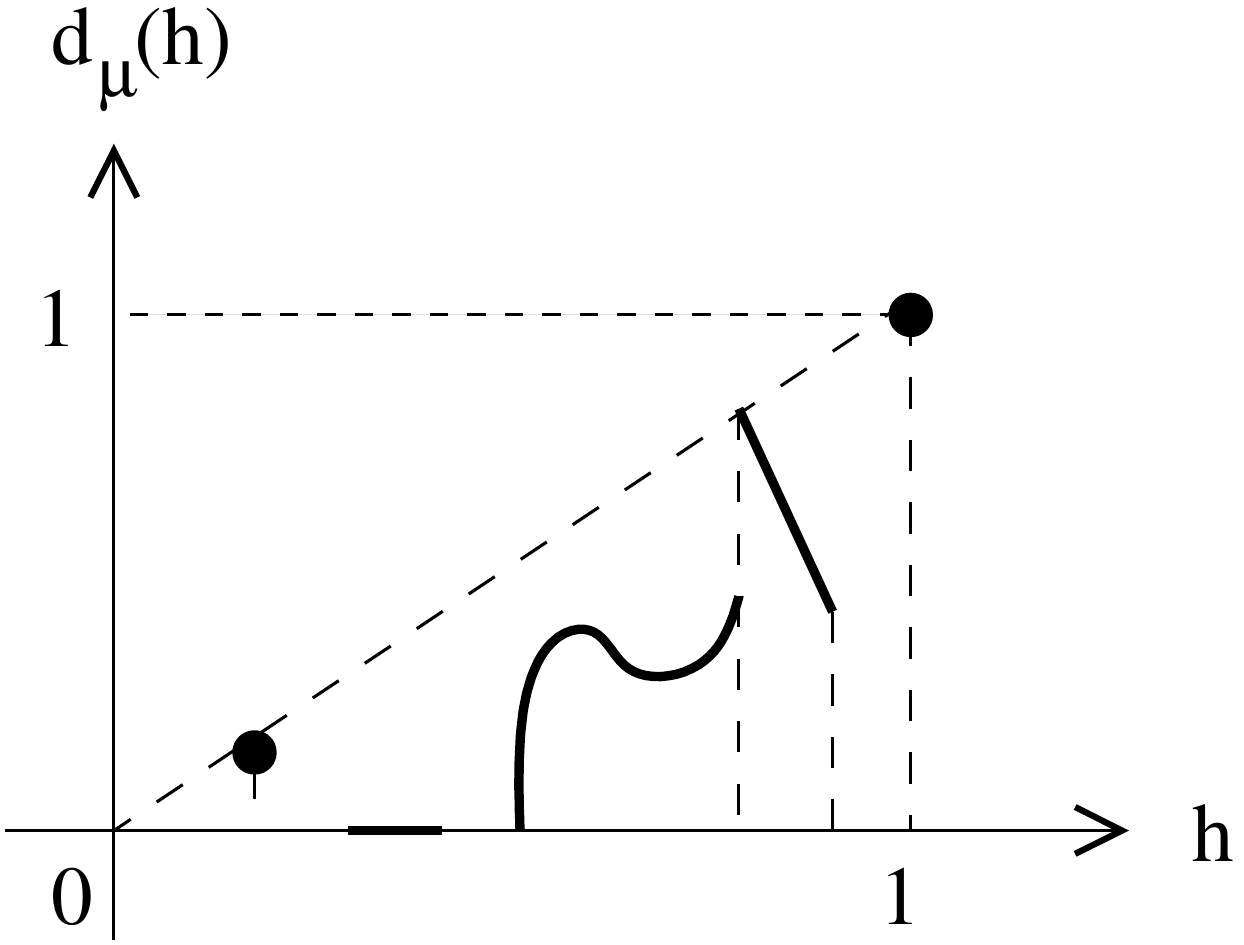}
  \includegraphics[width=5.5cm,height = 4.0cm]{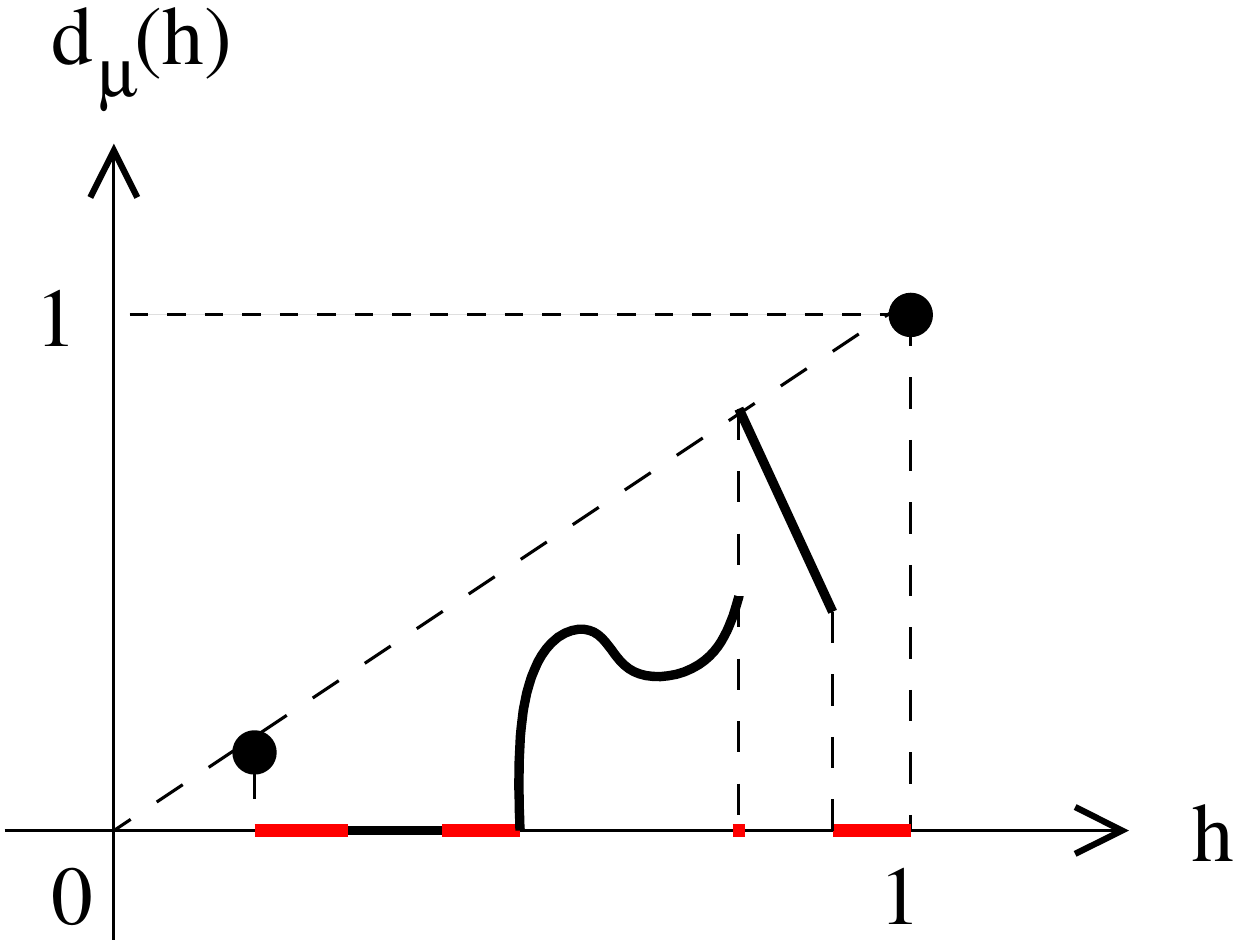}
\caption{Spectrum of a non-HM measure (left) and a HM measure (right).} \label{fig11}
\end{figure}
\end{center}
%%%%%%%%%%%%%%%%%%%%%%%%%%%%%%%%%% 

The proof of Theorem \ref{th11} is quite original. 
 The method is the application of an iterative algorithm which allows 
us to superpose various affine spectra and to create a HM measure $\mu$. Applying  this algorithm iteratively, we loose control on the  local dimensions of the limit measure $\mu$, but  only on  a zero-dimensional set  
  of exceptional
  points. It is really nice that Theorem \ref {thdarboux} asserts that this is  the normal situation (hence giving the optimality of Theorem \ref{th11}), since an uncountable set of points with new exponents   appears for HM  measures. Theorems \ref{thdarboux}, \ref{thdarbouxb} and Lemma \ref{*closedspectr} can also help us to verify that after this modification
we still have a HM measure.

\begin{remark}
Although the measures built in Theorem \ref{th11} have the same multifractal spectrum on any interval, they do not have to satisfy in general any kind of multifractal formalism (their spectrum has no reason to be concave). 
\end{remark}

\begin{remark}
\label{rk2}
The key point to prove Theorems \ref{th111} and  \ref{th11} is the explicit construction of a non-HM measure $\mu$ which has an affine spectrum, achieved in Theorem \ref{proplinearspectrum} in Section \ref{sec_cons1}.  {\bf Theorem \ref{proplinearspectrum} can be admitted at first reading, the other proofs use only the existence of such measures.}
\end{remark}

We have a similar result for the prescription of multifractal spectra of monotone functions.

%%%%%%%%%%%%%%%%%%%%%%%%%%%%%%
\begin{theorem}
\label{th1}
For every $f\in \mathcal{F}$, there exists a  HM strictly monotone 
increasing
continuous function $Z:\zu\to\R$    satisfying:
\begin{enumerate}
\sk
\item
for all $ h\in \mbox{Support}^*(f)$, $   d_Z(h) =  \max(f(h),0), $
\sk\item
 $d_{Z}(h) = -\infty $  if $h\in \R^+\sm$Support$^*(f)$,
\sk \item
 the set of points $\{x \in \zu:   h_Z(x)  =1\}$ has Lebesgue measure 1.
%\sk \item
% the set of points $\{x \in \zu:  h_Z(x)  >1 \}$ is empty. 
 \end{enumerate}
 \end{theorem}
%%%%%%%%%%%%%%%%%%%%%%%%%%%%%%

\mk

There is a subtle difference between Theorems \ref{th1} and \ref{th11} concerning the exponents larger   than 1.
This difference is due to the fact that \eqref{defpoint} eliminates the polynomial trends, while \eqref{defexpmu} does not: hence if a monotone function $Z$ is exactly linear with positive slope around a point $x$, one has $h_Z(x)=+\infty$, whilst for the corresponding measure $\mu$ (the derivative of $F$) one has $h_\mu(x)=1$.
Hence, to prove Theorem \ref{th1}, it is enough to   get rid of the points  $\{x \in \zu:  h_Z(x)  >1 \}$ by adding the
function constructed in the following proposition.

\begin{proposition}\label{homone}
There exists $Z:[0,1]\to[0,1]$ a
strictly monotone increasing HM function with $h_Z(x)  =1$ for all $x\in [0,1]$. 
\end{proposition}

It is also clear that Theorem \ref{th1} implies Theorem \ref{th11} if we consider
the measure $\mmm$ for which $\mmm((a,b))=(Z(b)-Z(a))/(Z(1)-Z(0))$,
for all $0\leq a<b\leq 1.$

\sk

Our final result concerns the prescription of the spectrum of HM non-monotone functions, which are more "flexible" than HM monotone functions. Combining   Theorem \ref{th1} with wavelet methods, we obtain the following result.

%%%%%%%%%%%%%%%%%%%%%%%%%%%%%%
\begin{theorem}
\label{th3}
Let $0<\alpha <\beta $ be two real numbers,    and consider $f\in \mathcal{F}$.
There exists a   continuous  HM  function $Z:\zu\to\zu$ satisfying:
 \begin{enumerate}
\item
for all $ h\in  \left( \alpha+(\beta-\alpha){Support}^*(f) \right)\setminus \{\beta\}$, $ \displaystyle d_Z(h)=   \max\left(f\left(\frac{h-\alpha}{\beta-\alpha}\right),0\right) $.
\sk\item
 $d_{Z}(h) = -\infty $  if $h\notin   \alpha+(\beta-\alpha){Support}^*(f) $,
\sk \item
 the set of points $\{x \in \zu:   h_Z(x)  =\beta\}$ has Lebesgue measure 1. 
  \end{enumerate}
   \end{theorem}
%%%%%%%%%%%%%%%%%%%%%%%%%%%%%%

The fact that the functions  $Z$ constructed in Theorem \ref{th3} are non-monotone allows one to take any support for the multifractal spectrum of $Z$. Theorem \ref{th3} is certainly far from being optimal, and the most general form of the multifractal spectrum of a function is an open issue. This leads to some open questions:

\mk

{\bf Question 1:} How to characterize and to prescribe local dimensions mapping $h\mapsto d_\mu(h)$ of probability measures $\mu$?  of HM probability measures $\mu$? 

 \mk

{\bf Question 2:} How to prescribe the multifractal spectrum with exponents larger than one for measures? Same questions for a HM spectrum?

\mk

{\bf Question 3:} Characterize  the most general form of the multifractal spectrum of a function, or of a measure? 
By Theorem \ref{thdarboux} it is not the  same 
  for functions and measures. Can one state something more precise?

\mk

{\bf Question 4:} Characterize  the most general form of the multifractal spectrum of a function, or of a measure satisfying a multifractal formalism?

\mk

{\bf Question 5:} What about  higher dimensional results?

\mk

The paper is organized as follows. 

Section \ref{sec_prel} contains preliminary results, and the proof of Lemma \ref{lem0}.

 In Section \ref {sec_prel2}, we prove  the Darboux Theorem \ref{thdarboux} for HM measures, using a maximal inequality result (Proposition \ref{*maxin}).

  Section \ref{sec_cons1} presents the construction of a  non-HM monotone function with an affine  increasing spectrum. As explained in Remark \ref{rk2},    the proof of  Theorem \ref{proplinearspectrum} can be omitted at first reading. Using this construction, we prove Theorem \ref{th111} in   Section \ref{sec_cons11}.
  
   In Section \ref{sec_cons3}, we  develop an iterative algorithm which creates a sequence of monotone functions converging  to a  HM  monotone function according to the constraints of Theorem  \ref{th11}.    
  
Using wavelet methods, non-monotone HM functions with    spectra  belonging to a dilated and translated version of the  function set $\mathcal{F}$ are built in Section \ref{sec_cons4}.

  Finally, the last Sections  \ref{sec_noncontinuous} and \ref{sechomone} contain the constructions of  a HM
measure with a gap in $ \mbox{Support\,($d_\mu$)}$ for exponents larger than one (Proposition \ref{*spectrgap}), and   of an 
increasing HM (monofractal) function  with $\ds h_{Z}(x)=1$ for all
$x\in [0,1].$ This yields Theorem \ref{th1}.

%%%%%%%%%%%%%%%%%%%%%%%%%%%%%%%%%%
%%%%%%%%%%%%%%%%%%%%%%%%%%%%%%%%%%
%%%%%%%%%%%%%%%%%%%%%%%%%%%%%%%%%%
%%%%%%%%%%%%%%%%%%%%%%%%%%%%%%%%%%
%%%%%%%%%%%%%%%%%%%%%%%%%%%%%%%%%%
%%%%%%%%%%%%%%%%%%%%%%%%%%%%%%%%%%
%%%%%%%%%%%%%%%%%%%%%%%%%%%%%%%%%% 
\section{Notation and preliminary results}
\label{sec_prel}

 The open ball centered at $x\in \R$ and of radius $r$ is denoted by $B(x,r)$. By $\overline{B(x,r)}$ we denote the closed ball.

 For a set $A\subset \R^d$
we denote its diameter by $|A|$ and by $\lll(A)$ its $d$-dimensional Lebesgue measure.  int($A)$ stands for the (open) interior of $A$.
 
The sum of two non-empty sets is  $A+B=\{ a+b: a\in A,\ b\in B \}$ for $A,B\sse \R^d$.

We refer to \cite{Fa1,F2,mattila} for the standard definition of the Hausdorff measure $\mathcal{H}^s(E)$ and Hausdorff dimension $\dim (E) $ of a set $E$. 

We will use the level sets $E_\mu({h}) $ of the \ho exponents \eqref{defeh}, but also the following  sets related to the \ho exponents:
\begin{equation}
E_Z^{\leq}(h) = \{x\in \zu: h_Z(x) \leq h\} \ \ \mbox{ and } \ \ E_\mu^{\leq}(h) = \{x\in \zu: h_\mu(x) \leq h\}.
\end{equation}

%%%%%%%%%%%%%%%%%%%%%%%%%%%%%%%%%%%
%%%%%%%%%%%%%%%%%%%%%%%%%%%%%%%%%%%
%%%%%%%%%%%%%%%%%%%%%%%%%%%%%%%%%%%
%%%%%%%%%%%%%%%%%%%%%%%%%%%%%%%%%%%
%%%%%%%%%%%%%%%%%%%%%%%%%%%%%%%%%%%
%%%%%%%%%%%%%%%%%%%%%%%%%%%%%%%%%%%

These sets arise regularly in our proofs. Standard results on  multifractal spectra of monotone functions and measures give upper bounds for their Hausdorff dimension.

%%%%%%%%%%%%%%%%%%%%%%%%%%%%%%%%%%
\begin{proposition} {\cite{BMP}}
\label{prop2}
Let $\mu$ be a Borel probability measure on $\zu$. 

 \noindent
 For every $h\in \zu$, $\dim E_\mu^{\leq}(h) \leq h$.

 \noindent
 In particular, for every $h\in \zu$, $d_\mu(h) = \dim E_\mu({h}) \leq h$.
 
 \noindent
 The same holds for $ E_Z^{\leq}(h)$ and $ E_Z({h})$ for any monotone function $Z:\zu\to\R$.   
 \end{proposition}
 
%%%%%%%%%%%%%%%%%%%%%%%%%%%%%%%%%%

From this proposition one deduces  Lemma \ref{lem0}.

\begin{proof}[Proof of Lemma \ref{lem0}] %{\em [of Lemma \ref{lem0}]}
Assume that the mapping $x\mapsto h_\mu(x)$ is continuous, and that $h_\mu(x_0)\neq d$ for an $\xo \in \OOO$. 
By continuity, for some constants $0<\ep<M$, one has  $ \ep \leq |h_\mu(x) - d| \leq M$ for all $x$ in  an open ball $B\sse \OOO$  around $x_0$ such that $\mu(B)>0$.

If $h_\mu(x)\leq d-\ep$ when $x\in B$, then $\dim_H(\{x\in B: h_\mu(x)\leq d-\ep\}) =d$, which is impossible by Propostion \ref{prop2}.

If $h_\mu(x) \geq d+\ep$ when $x\in B$, then the argument is as follows. Fix $\eta>0$. For every $x\in B$, there exists $0<r_x <\eta$ such that  $B(x,r_x) \subset B$ and $\mu(B(x,r_x)) \leq |B(x,r_x)|^{d+\ep/2} $. Hence, $\{B(x,r_x)\}_{x\in B}$ forms a covering of $B$ by balls centered at points of $B$. By Besicovitch's 
covering theorem, there is an integer $Q'(d)$ depending only on the dimension $d$ such that one can find $Q'(d)$  families $\mathcal{F}_i$,  $i=1,..., Q'(d)$ of disjoint balls amongst the balls $B(x,r_x)$ such that
$$B\subset \bigcup_{i=1,..., Q'(d)} \ \bigcup_{B'  \in \mathcal{F}_i} B'.$$
All the balls within one $\mathcal{F}_i$ are disjoint, hence (using that   $\lambda (B') =   C_d2^{-d} |B'|^d$  for any ball $B'\subset \R^d $, where $C_d$ is  the volume of the unit ball in $\R^d$) 
\begin{eqnarray*}
  \mu (B)  & \leq  &  \sum_{i=1}^{Q'(d)}  \mu(\bigcup_{B'  \in \mathcal{F}_i}  B') = \sum_{i=1}^{Q'(d)}   \sum_{B'  \in \mathcal{F}_i}  \mu(B') \leq    \sum_{i=1}^{Q'(d)}  \sum  _{B'  \in \mathcal{F}_i}  |B'| ^{d+\ep/2 } \\
  & \leq &   \eta^{\ep/2} \sum_{i=1}^{Q'(d)}  \sum  _{B'  \in \mathcal{F}_i}  |B'| ^{d }  \leq \frac{\eta^{\ep/2}2^d}{C_d} \sum_{i=1}^{Q'(d)}  \sum  _{B'  \in \mathcal{F}_i}  \lambda(B' )  \leq   \frac{\eta^{\ep/2}2^d}{C_d}  Q'(d) \lambda(B).
\end{eqnarray*}
Letting $\eta$ tend to zero we obtain $\mu(B)=0$, which is impossible.\end{proof}

%%%%%%%%%%%%%%%%%%%%%%%%%%%%%%%%%%
%%%%%%%%%%%%%%%%%%%%%%%%%%%%%%%%%%
%%%%%%%%%%%%%%%%%%%%%%%%%%%%%%%%%%
%%%%%%%%%%%%%%%%%%%%%%%%%%%%%%%%%%
%%%%%%%%%%%%%%%%%%%%%%%%%%%%%%%%%%
%%%%%%%%%%%%%%%%%%%%%%%%%%%%%%%%%%
%%%%%%%%%%%%%%%%%%%%%%%%%%%%%%%%%% 
\section{The Darboux theorem for  HM  measures}
\label{sec_prel2}

%%%%%%%%%%%%%%%%%%%%%%%%%%%%%%%%%%
%%%%%%%%%%%%%%%%%%%%%%%%%%%%%%%%%%
%%%%%%%%%%%%%%%%%%%%%%%%%%%%%%%%%%
%%%%%%%%%%%%%%%%%%%%%%%%%%%%%%%%%%
\subsection{A maximal inequality}

We need a variant of the maximal inequality used for the deduction
of the Hardy--Littlewood maximal inequality. In the sequel, we  always consider the case $d=1$, but the next proposition holds in any finite dimension, $d$.
%%%%%%%%%%%%%%%%%%%%%%%%%%%%%%%%%% 
\begin{proposition}\label{*maxin}
Suppose that $I\sse [0,1]^d$ is an open ball, $\mmm$ is a measure on $[0,1]^d$, $0<\bbb\leq 1$
and set
$$\mcsib \, \mmm(x) :=  \sup\Big \{ \frac{\mmm(B(x,r))}{(2r)^{d\bbb}}: r>0,\ B(x,r)\sse I \Big \}.$$
Then, there exists a constant $Q(d)>0$  (depending on the dimension $d$  only) such that  for all $t>0$, the set  $\mcst = \{ x \in \zu^d:\mcsib \, \mmm(x)>t \}$ satisfies  
\begin{equation*}
\lll(\mcst)   \leq \frac{Q(d) \, \mmm(I)  \, |I|^{d(1-\bbb)}}{t}.
\end{equation*}
\end{proposition}
%%%%%%%%%%%%%%%%%%%%%%%%%%%%%%%%%% 
 
%%%%%%%%%%%%%%%%%%%%%%%%%%%%%%%%%% 
\begin{proof}
For every $x\in \mcst$, there is $r_x$ such that $\mu(B(x,r_x))\geq t (2r)^{d\beta}$. Hence the family of balls $\{B(x,r_x)\}_{x\in \mcst}$ forms a covering of $\mcst$ by balls centered at points of $\mcst$. Recall that for a ball $B\subset \R^d$ of radius $r$, $\lambda(B)=C_d r^d = C_d2^{-d} |B|^d$. Using  Besicovitch's covering theorem, there exists a constant $Q'(d)>0$ (depending only on the dimension) such that one can extract from this (possibly uncountable) family   a countable system $B(x_{i},r_{i})\sse I$, $x_{i}\in\mcst$
such that:
\begin{itemize}
\item 
the union of these balls covers  $\mcst$,
\item
no point $x\in\R^d$ is covered by more than $Q'(d)$ balls of the form $B(x_{i},r_{i})$,
\item  for every $i$, we have  
\begin{eqnarray}
\nonumber
\mmm(B(x_{i},r_{i})) & > & t(2r_{i})^{ d\bbb}  = t \cdot 2^{d\beta} \cdot  (C_d)^{-\beta}  \cdot \lambda(B(x_i,r_i))^\beta\\
\nonumber& =  & t \cdot 2^{d\beta} \cdot  (C_d)^{- \beta}  \cdot \lambda(B(x_i,r_i))  \cdot \lll(B(x_{i},r_{i}))^{\bbb-1}\\
\nonumber& \geq&
t\cdot    2^{d\beta} \cdot  (C_d)^{- \beta}  \cdot \lambda(B(x_i,r_i))  \cdot  \lambda(I)^{\bbb-1}\\
\label{*T1*2}
& \geq&
t\cdot    2^{d } \cdot  (C_d)^{- 1}  \cdot \lambda(B(x_i,r_i))  \cdot |I|^{d(\bbb-1)}.
\end{eqnarray} 
\end{itemize}
Since no point is covered by more than $Q'(d)$ balls $B(x_{i},r_{i})$, one can select an index set
$\cai$ such that the balls $B(x_{i},r_{i})$, $i\in \cai$ are disjoint and the Lebesgue measure of $\cup_{i\in\cai}B(x_{i},r_{i})$
 is greater than $1/Q'(d)$ times that of  $\mcst$. Then summing \eqref{*T1*2} for $i\in\cai$, one obtains
$$\frac{1}{Q'(d)}\lll(\mcst)\leq 
\sum_{i\in \cai}\lll(B(x_{i},r_{i}))\leq 2^{-d}  C_d
|I|^{d(1-\bbb)}\frac{1}{t}\sum_{i\in\cai}\mmm(B(x_{i},r_{i}))\leq
\frac{2^{-d}  C_d\mmm(I)|I|^{1-\bbb}}{t}.$$
Hence the result with $Q(d)=Q'(d)^{-1}2^{-d}C_d$.
\end{proof}
%%%%%%%%%%%%%%%%%%%%%%%%%%%%%%%%%% 

%%%%%%%%%%%%%%%%%%%%%%%%%%%%%%%%%%
%%%%%%%%%%%%%%%%%%%%%%%%%%%%%%%%%%
%%%%%%%%%%%%%%%%%%%%%%%%%%%%%%%%%%
%%%%%%%%%%%%%%%%%%%%%%%%%%%%%%%%%%
\subsection{The minimum property of the spectrum for non atomic HM measures.}

\begin{lemma}\label{*closedspectr}
Suppose that $\mmm$ is a non-atomic measure supported on
$[0,1]$ and $0\leq \aaa <1$. 
Assume that:
\begin{itemize}
\item for every $\eee>0$,  $E_\mu^{\leq} (\aaa+\eee )$
is dense in $[0,1]$,
\item
 $h_{\mmm}(x)\geq \aaa$ for all $x\in[0,1]$.
\end{itemize}
Then $E_\mu(\alpha)   $ is dense in $[0,1]$. 
\end{lemma}

Observe  that Lemma \ref{*closedspectr} implies that for a non-atomic HM measure, 
if $\aaa_{0}=\inf \{ h_{\mmm}(x):x\in [0,1] \}<1$, then necessarily $d_{\mmm}(\aaa_{0})\geq 0.$

\begin{proof}
Suppose $0\leq a <b\leq 1$.
Choose $x_{1}\in (a,b)$ and $r_{1}\in (0,1)$  such that $\mmm(B(x_{1},r_{1}))>r_{1}^{\aaa+\frac{1}{2}}$.
Since $\mmm$ is non-atomic one can choose a small non-degenerate closed
interval $I_{1}\sse (a,b)$  such that $x_{1}\in I_{1}$ and
$$\text{for any }x\in I_{1}\text{ we have  }\mmm(B(x,r_{1}))>r_{1}^{\aaa+\frac{1}{2}}.$$
Suppose that we have defined the non-degenerate nested intervals
$I_{1}\supset I_{2}\supset ... \supset I_{n}$ and $r_{1}>r_{2}>...>r_{n}>0$
 such that 
 \begin{equation}\label{*intest}
 \text{for any }x\in I_{n}\text{ we have  }\mmm(B(x,r_{n}))>r_{n}^{\aaa+\frac{1}{n+1}}\text{ and }r_{n}<\frac{1}{n}.
 \end{equation}
By our assumption we can choose $x_{n+1}\in \intt(I_{n})$
and $r_{n+1}\in (0,\frac{1}{n+1})$  such that $$\mmm(B(x_{n+1},r_{n+1}))>r_{n+1}^{\aaa+\frac{1}{n+2}}.$$
One can also choose a non-degenerate closed interval $I_{n+1}\sse I_{n}$
 such that 
 $$\text{for any }x\in I_{n+1}\text{ we have  }\mmm(B(x,r_{n+1}))>r_{n+1}^{\aaa+\frac{1}{n+2}}\text{ and }r_{n+1}<\frac{1}{n+1}.$$
 By induction we can define the infinite nested sequence of intervals 
 $I_{n}$  such that \eqref{*intest}
holds for all $n$. Then any $x_{\aaa}\in \cap_{n=1}^{\oo}I_{n}\sse (a,b)$
satisfies $h_{\mmm}(x_{\aaa})=\aaa.$
 \end{proof}

%%%%%%%%%%%%%%%%%%%%%%%%%%%%%%%%%%
%%%%%%%%%%%%%%%%%%%%%%%%%%%%%%%%%%
%%%%%%%%%%%%%%%%%%%%%%%%%%%%%%%%%%
%%%%%%%%%%%%%%%%%%%%%%%%%%%%%%%%%%
\subsection{Continuity of the support of the multifractal spectrum, $  \mbox{Support\,($d_\mu$)}$}

Theorem \ref{thdarboux}  is a direct consequence of the next result.

%%%%%%%%%%%%%%%%%%%%%%%%%%%%%%%%%%
\begin{theorem}\label{thdarbouxb}
Let $\mu$ be a non-atomic  Borel probability measure  supported in the interval $\zu$. Assume that there exists $0\leq \alpha<1$ such that for every $\eee>0$,  $\{ x: h_{\mmm}(x) \leq \aaa+\eee \}$ is dense in $[0,1]$ and $h_{\mmm}(x)\geq \aaa$ for all $x\in [0,1]$. Then for every $\beta \in [\alpha,1]$, $d_\mu(\beta)\geq 0$ and $d_{\mmm}(\bbb)=-\oo$ for $\bbb<\aaa$.
\end{theorem}
%%%%%%%%%%%%%%%%%%%%%%%%%%%%%%%%%%

%%%%%%%%%%%%%%%%%%%%%%%%%%%%%%%%%%
%%%%%%%%%%%%%%%%%%%%%%%%%%%%%%%%%%
\begin{proof}
The function $F_{\mmm}(x)=\mu([0,x])$ is 
continuous, monotone increasing  and hence
$\lll$-almost everywhere differentiable by Lebesgue's theorem.
We denote by $D_{\mmm}$ the set of those points $x\in(0,1)$ where (a finite) $F_{\mmm}'(x)$ exists.  

\medskip

The case $\bbb=\aaa$ is treated in Lemma \ref{*closedspectr}.
The case $\bbb=1$ is slightly different and will be discussed later. 

\medskip

Suppose $\aaa<\bbb<1$.  We are going to build iteratively a nested sequence of intervals converging to one point $x$ such that $h_\mu(x)=\beta$. 

Put $\bbb'=\frac{1+\bbb}{2}$ (so that $ \beta <\beta'<1$). Clearly, for any $x\in D_{\mmm}$  there exists $r_{\bbb,x}$
 such that $\mmm(B(x,r))\leq r^{\bbb'}$ for all $0<r<r_{\bbb,x}<1.$
 
Choose an $x_{0}\in D_{\mmm}$ and suppose that $r_{0} \leq r_{\beta,x_0}$  
satisfies 
\begin{equation}\label{*T3*1}
\mmm(B(x_{0},r_{0}))\leq r_{0}^{\bbb'}  \  \text{ and }   \ r_{0}^{\bbb'-\bbb}<\frac{1}{10}.
\end{equation}

%%%%%%%%%%%%%%%%%%%%%%%%%%%%%%%%%%
We start the induction. 

Let $n\geq 1$ be fixed.
Assume that one can construct   two sequences of positive real numbers  $\{(x_{0},r_0), (x_1,r_1),...,(x_{n},r_n)\}$  and $\{(\tx_{0},\tr_0), (\tx_1,\tr_1), ...,(\tx_{n-1},\tr_{n-1})\}$ satisfying  the following properties:

\begin{itemize}
\sk
\item [(P0)]
the real numbers $x_0,x_1,...x_n, \tx_0,\tx_1,..., \tx_{n-1}$ belong to the interval $\zu$,

\sk
\item[(P1)] the radii  are decreasing with $ n$, and they satisfy
$$r_{0}>\tr_{0}>r_{1}>\tr_{1}>...>\tr_{n-1}>r_{n}>0,$$

\sk
\item[(P2)] 
 for $i=1,...,n$, one has
\begin{eqnarray}\label{*T3*2}
\overline{B(x_{i},r_{i})}  & \sse  & B(\tx_{i-1},\tr_{i-1}/3) \, \sse  \, B(x_{i-1},r_{i-1}/3), \\
\label{*T3*3}
\mmm(B(x_{i},r_{i})) & \leq  & r_{i}^{\bbb'}  ,\\
\label{*T3*33}
 \mmm(B(\tx_{i-1},\tr_{i-1}))  & =  & \tr_{i-1}^{\bbb},
\end{eqnarray}

\sk

\item[(P3)] 
 if $x\in\overline{B(x_{i},r_{i}/3)}$ for some $i=1,...,n$,  and if  $B(x,r)\not\sse B(x_{i},r_{i})$
but $B(x,r)\sse B(x_{i-1},r_{i-1})$, then
\begin{equation}\label{*S5*2}
\mmm(B(x,r))<300\cdot r^{\bbb}.
\end{equation}
\end{itemize}

\medskip

Observe that (P2)  implies that for any $x\in\overline{B(x_{i},r_{i})}$
we have 
\begin{equation}\label{*S5*1}
\mmm(B(x,2\tr_{i-1}))\geq \mmm(B(\tx_{i-1},\tr_{i-1}))=\tr_{i-1}^{\bbb}.
\end{equation}

%%%%%%%%%%%%%%%%%%%%%%%%%%%%%%%%%% 
\begin{center}
\begin{figure}
  \includegraphics[width=10.0cm,height = 5.5cm]{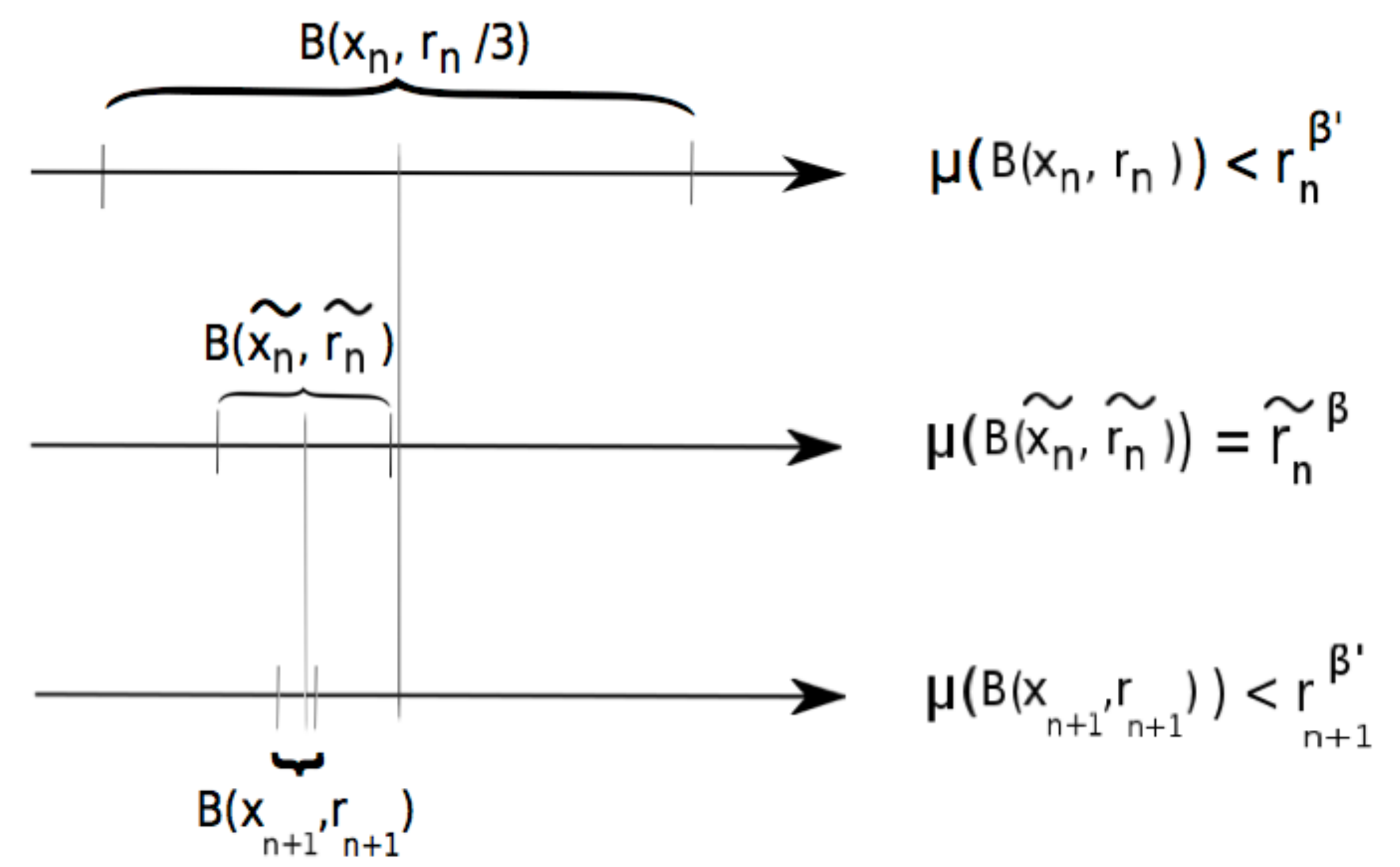}
\caption{Nested sequence of balls in the construction} \label{fig12}
\end{figure}
\end{center}
%%%%%%%%%%%%%%%%%%%%%%%%%%%%%%%%%% 

Let us explain why such a sequence is key to prove Theorem  \ref{thdarbouxb}.

\begin{lemma}
\label{lem0000}
If  there exist four infinite sequences  $(x_{n})$, $(r_{n})$, $(\tx_{n})$ and $(\tr_{n})$
satisfying (P0-3)  for all $n\in\N$, then   Theorem \ref{thdarbouxb} is proved.
\end{lemma}
\begin{proof}
The sequences of radii obviously converge to zero by (P1) and (P2). 

Let $\{x\}=\bigcap_{i=1}^{\oo}\overline{B(x_{i},r_{i})}$.   The "nesting" relation  \eqref{*T3*2} implies that $x$ is always located in the "middle part" of the balls $B(x_i,r_i)$, more precisely in $B(x_i,r_i/3)$ and $B(\tx_i,\tr_i/3)$. 
 
By definition of $h_\mu(x)$, and using \eqref{*S5*1}, we obtain
\begin{equation}
\label{eqqqqq}
h_\mu(x)  \leq \liminf_{i\to+\infty} \frac{\log \mu(B(x,2\tr_i))}  {\log |B(x,2\tr_i)| } \leq \beta.
\end{equation}

Property (P3) allows us to control the behavior of $\mu(B(x,r)) $, for every $r\in (0,r_{1}/3)$. 
Indeed,  fix such a radius $r$, and choose $i\geq 1$ as the unique integer  such that $r_{i}/3< r \leq r_{i-1}/3$ . By construction, one has  $x\in B(x_i,r_i/3)$. Let $R $ be the largest positive real number such that $B(x,R) \subset B(x_i,r_i)$. Obviously, one has  $2r_i/3\leq R \leq r_i \leq r_{i-1}/3$.

\begin{itemize}
\item
if $2r_i/3 < r \leq R$: then using \eqref{*T3*3}
$$\mu(B(x,r)) \leq \mu(B(x_i,r_i)) \leq (r_i)^{\beta'} \leq 3^{\beta'} r^{\beta'}.$$

\item
if $R \leq  r \leq r_{i-1}/3$: then we are in the situation described by (P3), since $B(x,r) \not\subset B(x_i,r_i)$ but from $x\in B(x_{i-1},r_{i-1}/3)$
it follows that  $B(x,r)\subset B(x_{i-1},r_{i-1})$.
 Hence \eqref{*S5*2} holds true.

\smallskip
\end{itemize}
In any case, one sees that $\mu(B(x,r))\leq C r^{\beta}$  for any $r \leq r_1$, for some constant $C$. This implies that $h_\mu(x)\geq \beta$. Combining this with \eqref{eqqqqq}, Lemma \ref{lem0000} is proved.
\end{proof}

\medskip

It remains us to prove that   as soon as we
are given
$(x_{i}$, $r_{i})$ for $i=0,...,n$ and $(\tx_{i}$, $\tr_{i})$ for $i=1,...,n-1$ fulfilling properties (P0-3), we can  construct $(x_{n+1},\ r_{n+1}),\ (\tx_{n},\tr_{n})$ satisfying the same properties.

\medskip

From $r_{n}\leq r_{0}$, \eqref{*T3*1}, and \eqref{*T3*3}, one deduces that
\begin{equation}\label{*T5*1}
\mmm(B(x_{n},r_{n}))\leq r_{n}^{\bbb'} \ \ \text{ and } \ \ r_{n}^{\bbb'-\bbb}<\frac{1}{10}.
\end{equation}
By our assumption  the set $\{ x:h_{\mmm}(x)<\bbb \}$
is dense in $[0,1]$. Choose $\txn\in B(x_{n},r_{n}/6)$  such that 
$h_{\mmm}(\txn)< \bbb$.

 %%%%%%%%%%%%%%%%%%%%%%%%%%%%%%%%%%
\begin{lemma}
There exists a   largest  real number  $\trn$ such that $0<\trn <r_{n}$   and \begin{eqnarray}\label{*S1*2}
\mmm(B(\txn,\trn)) & = &  (\trn)^{\bbb} \\  
\label{*S1*22}
\text{ if  }\trn<r \mbox{ and } 
B(\txn,r) & \sse  & B(x_{n},r_{n}),   \mbox{ then } \ \mmm(B(\txn,r))<r^{\bbb}.
\end{eqnarray}
Moreover,   one  necessarily  has $\trn<r_{n}/10$.
\end{lemma}
%%%%%%%%%%%%%%%%%%%%%%%%%%%%%%%%%%

%%%%%%%%%%%%%%%%%%%%%%%%%%%%%%%%%%
\begin{proof}
If $r\geq r_n/10$ and $B(\txn,r)\sse B(x_{n},r_{n})$ then 
from \eqref{*T5*1} 
we deduce  that 
$$\mmm(B(\txn,r))\leq \mmm(B(x_{n},r_{n}))\leq
r_{n}^{\bbb'} \leq \frac 1{10}  r_n ^{\beta} \leq  \frac{10^\beta}{10} r^\beta < r^\beta.$$ 
Hence, if there exists a suitable  $\trn$, then $\trn<r_n/10$.
 
Since $h_\mu(\tilde x_n)<\beta$, there exists $0<r<r_n/20$ such that $\mu(B(\txn,r)) >r^\beta$.
By continuity  of the map $r\mapsto \mu(B(\txn,r))$, we can choose $\trn$ as the largest $r$ such that \eqref{*S1*2} is satisfied. This choice of $\trn$ implies relation \eqref{*S1*22}.
\end{proof}
%%%%%%%%%%%%%%%%%%%%%%%%%%%%%%%%%%

Observe that, since $\txn\in B(x_{n},r_{n}/6)$ and $\tr_n<r_n/10$, one has
\begin{equation}\label{*S2*1}
B(\txn,\trn)\sse B(x_{n},r_{n}/3).
\end{equation}

 %%%%%%%%%%%%%%%%%%%%%%%%%%%%%%%%%%
\begin{lemma}\label{haromhat}
Let $x\in \overline{B(\txn,\trn/3)}$, $r\geq \trn/3$
and assume that $B(x,r)\sse B(x_{n},r_{n})$. 
Then we have 
\begin{equation}\label{*S3*2}
\mmm(B(x,r))< 5 r^{\bbb}.
\end{equation}
\end{lemma}
 %%%%%%%%%%%%%%%%%%%%%%%%%%%%%%%%%%

 %%%%%%%%%%%%%%%%%%%%%%%%%%%%%%%%%%
\begin{proof}
By construction,  $B(x,r)\sse B(\txn,4r)$.

\sk

$\bullet$ If $B(\txn,4r)\sse B(x_{n},r_{n})$:  using \eqref{*S1*22}, we obtain
\begin{equation*}
\mmm(B(x,r))\leq \mu(B(\txn,4r))\leq (4r)^{\bbb}\leq 4 r^{\bbb}.
\end{equation*} 

\sk

$\bullet$ 
If $B(\txn,4r)\not \sse B(x_{n},r_{n})$: the fact that  $\txn\in B(x_{n},r_{n}/6)$ implies that $4r>5r_{n}/6$, that is, $r>5r_{n}/24$, and
\begin{equation}\label{*S3*1}
\mmm(B(x,r))\leq \mmm(B(x_{n},r_{n}))\leq r_{n}^{\bbb'}\leq
r_{n}^{\bbb}<(24/5)^{\bbb}r^{\bbb}<5r^{\bbb}.
\end{equation}
\end{proof}
 %%%%%%%%%%%%%%%%%%%%%%%%%%%%%%%%%%

 Set $I_{n+1}=B(\txn,\trn)$
and $\hI_{n+1}=B(\txn,\trn/3)$.
 Now, apply Proposition  \ref{*maxin}
with $I=I_{n+1}$ and  $t=50$. 
Since $\mmm(I_{n+1})=\mmm(B(\txn,\trn))=(\trn)^{\bbb}$, 
we have 
$$\lll  \big(  \big \{ x:M^{*}_{I_{n+1},\bbb} \, \mmm(x)>50 \big \}  \big )\leq \frac{Q(1) \cdot ( \trn)^{\bbb}(2\trn)^{1-\bbb}}{50}<\lll(\hI_{n+1}). $$
For the last inequality, we have used that the constant $Q(1)$ in Proposition \ref{*maxin}  is less than 5, \cite{mattila}. Hence, recalling that $D_\mu$ has full Lebesgue measure in  $\hI_{n+1}$,   we can pick 
$x_{n+1}\in \hI_{n+1}\cap D_{\mmm}=B(\txn,\trn/3)\cap D_{\mmm}
$  such that for all $r>0$ if $B(x_{n+1},r)\sse B(\txn,\trn)$
then 
\begin{equation}\label{*S2*2}
\mmm(B(x_{n+1},r))\leq 50(2r)^{\bbb}\leq 100 r^{\bbb}.
\end{equation}

Using that $x_{n+1}\in D_{\mmm}$, one can also   choose $0<r_{n+1}<\tr_{n}/100$
 such that 
 \begin{equation}\label{*S3*3}
 \overline{B(x_{n+1},r_{n+1}) }\sse B(\txn,\trn/3)\sse B(x_{n},r_{n}/3)\
 \text{ and }\
 \mmm(B(x_{n+1},r_{n+1}))\leq r_{n+1}^{\bbb'}.
 \end{equation}

 %%%%%%%%%%%%%%%%%%%%%%%%%%%%%%%%%%
\begin{lemma}
 For all $x\in \overline{B(x_{n+1},r_{n+1}/3)}$,  if
$B(x,r)\not \sse B(x_{n+1},r_{n+1})$ but $B(x,r)\sse B(x_{n},r_{n})$,
then
\begin{equation*}
\mmm(B(x,r))<300r^{\bbb}.
\end{equation*}
\end{lemma}
 %%%%%%%%%%%%%%%%%%%%%%%%%%%%%%%%%%

 %%%%%%%%%%%%%%%%%%%%%%%%%%%%%%%%%%
\begin{proof}
The case $r\geq \trn/3$ is a consequence of Lemma \ref{haromhat}.
 
Let $x$ be as in the statement, and assume that   $r<\trn/3$.

From $B(x,r)\not \sse B(x_{n+1},r_{n+1})$, we deduce that  $2r_{n+1}/3<r.$ 
Hence, $B(x,r)\sse B(x_{n+1},3r)$  and $B(x,r)\sse B(\txn,\trn)$.

\sk

$\bullet$ If $B(x_{n+1},3r)\sse B(\txn,\trn)$ then  by \eqref{*S2*2}, we have
\begin{equation*}
\mmm(B(x,r))\leq \mmm(B(x_{n+1},3r))\leq 100\cdot  3^{\bbb}r^{\bbb}<
300 r^{\bbb}.
\end{equation*}

\sk

$\bullet$ If $B(x_{n+1},3r)\not \sse B(\txn,\trn)$ then one necessarily has  $3r>2\trn/3$.
Since 
$r<\trn /3$ and
$\trn< r_{n}/10$,  we have the inclusions
$B(x,r)\sse B(x,\trn/3)\sse B(x_{n},r_{n})$. Finally,  by
\eqref{*S3*2} used with $r=\trn/3$, we infer
\begin{equation*}
\mmm(B(x,r))\leq \mmm(B(x,\trn/3))\leq 5(\trn/3)^{\bbb}<
5(9r/6)^{\bbb}<25r^{\bbb}.
\end{equation*}
\end{proof}
 %%%%%%%%%%%%%%%%%%%%%%%%%%%%%%%%%%

Summarizing the above, (P0-3) hold when $i=n+1$. Iterating the procedure and applying the three above technical lemmas, we complete our inductive construction.

\medskip

Finally,  we discuss the case  $\bbb=1$,    indicating the minor
adjustments in the proof.

If  there exists $x\in D_{\mmm}$ such that $F_{\mmm}'(x)>0$,
then 
$$h_{\mmm}(x)=\liminf_{r\to 0}\frac{\log\mu(B(x,r))}{\log 2r} =\liminf_{r\to 0}\frac{\log |F_\mu(x+r)-F_\mu(x-r)|}{\log 2r} = 1 ,$$ 
and we are done.
Hence, we suppose that $F_{\mmm}(x)'=0$ for all $x\in D_{\mmm}$.

Choose an $x_{0}\in D_{\mmm}$ and $r_{0}$  such that 
instead of \eqref{*T3*1} we have
\begin{equation*}
\mmm(B(x_{0},r_{0}))\leq r_{0}/20 .
\end{equation*}

Assume that $n\geq 0$ and $(x_{0},r_0)$, ..., $(x_{n},r_n),$ $(\tx_{1},\tr_1),  ...,(\tx_n,\tr_{n})$   are given 
as before and  that \eqref{*T3*2} holds.
Instead of \eqref{*T3*3} and \eqref{*T3*33} we have
\begin{equation*}
\mmm(B(x_{i},r_{i}))\leq r_{i}/20 \text{ and }
\mmm(B(\tx_{i-1},\tr_{i-1}))=\tr_{i-1}\text{ for }i=1,...,n.
\end{equation*}
This implies that  
\eqref{*S5*1}  and \eqref{*S5*2} hold  true 
with $\bbb=1$.  
Further, instead of \eqref{*T5*1} we use
\begin{equation}\label{*Z2*1}
\mmm(B(x_{n},r_{n}))\leq r_{n}/20.
\end{equation}
We select $\txn $ and $\trn$ such that \eqref{*S1*2}
holds with $\bbb=1$.

By \eqref{*Z2*1}, since $B(\txn,r)\sse B(x_{n},r_{n})$, one deduces that 
 $\mmm(B(\txn,r))\leq \mmm(B(x_{n},r_{n}))\leq
r_{n}/20.$ Hence, for $r\geq r_{n}/10$ we have
$\mmm(B(\txn,r))<r$. Therefore,
\eqref{*S2*1} holds in this case as well.

 Finally, Lemma \ref{*maxin} used 
with $t=50$ allows us to  select some real number $x_{n+1}$, and we keep on  arguing as before. One only needs to remove
the inequality containing   $r_{n}^{\bbb'}$ from
\eqref{*S3*1} and
in \eqref{*S3*3} instead of $r_{n+1}^{\bbb'}$
we have to use $r_{n+1}/20$.
We also have to keep in mind that $x_{n+1}\in D_{\mmm}$
and our assumption implies $F'_{\mmm}(x_{n+1})=0.$
\end{proof}
%%%%%%%%%%%%%%%%%%%%%%%%%%%%%%%%%%

\begin{remark}
It is not difficult to modify the above proof 
so that at each step, in    $\ds B(x_{n},r_{n})$,
two balls $\ds B(x_{n+1}',r_{n+1}')$ and $B(x_{n+1}'',r_{n+1}'')$
are found with the required properties.
Iterating this remark, one concludes that  $E_\mu(\beta)$ is    uncountable.
\end{remark}

%%%%%%%%%%%%%%%%%%%%%%%%%%%%%%%%%%
%%%%%%%%%%%%%%%%%%%%%%%%%%%%%%%%%%
%%%%%%%%%%%%%%%%%%%%%%%%%%%%%%%%%%
%%%%%%%%%%%%%%%%%%%%%%%%%%%%%%%%%%
%%%%%%%%%%%%%%%%%%%%%%%%%%%%%%%%%%
%%%%%%%%%%%%%%%%%%%%%%%%%%%%%%%%%%
%%%%%%%%%%%%%%%%%%%%%%%%%%%%%%%%%%
%%%%%%%%%%%%%%%%%%%%%%%%%%%%%%%%%%
%%%%%%%%%%%%%%%%%%%%%%%%%%%%%%%%%%
%%%%%%%%%%%%%%%%%%%%%%%%%%%%%%%%%%
%%%%%%%%%%%%%%%%%%%%%%%%%%%%%%%%%% 

\section{A  non-HM monotone function  with an affine   spectrum    }

\label{sec_cons1}

The function ${\bf 1}\!\!  \!{\bf 1}^{*}_{[\aaa_{0},\bbb_{0}]}(h)$
equals $1$ if  $h\in [\aaa_{0},\bbb_{0}]$ and equals $-\oo$ otherwise.
In this section, we  work with monotone functions rather than measures: although the result is the same at the end, functions are more convenient to deal with in our construction.

We construct a monotone function (equivalently a measure) whose spectrum is affine on an interval strictly included in $\zu$ and compatible with the conditions of a spectrum (Proposition \ref{prop2}). In the next sections, we explain how the superposition of   functions built in Theorem \ref{proplinearspectrum}   yield HM and non-HM measures with prescribed spectrum.

%%%%%%%%%%%%%%%%%%%%%%%%%%%%%%%%%% 
\begin{theorem}
\label{proplinearspectrum}
Let $0<\alpha_0\leq \beta_0 < 1$.  Let $0< d<\alpha_0$ and $\eta>0$ 
satisfy 
\begin{equation}\label{*4.1*1}
\text{ $d(1+\eta\beta_0) \leq \beta_0$ and  $d(1+\eta\aaa_0) \leq \aaa_0$}.
\end{equation}

Then there exists a monotone continuous function $Z $  with the following properties:  $Z(x)=0$ when $x\leq 0$,  $Z(x)=1$ when $x\geq 1$, $d_Z(+\infty)=1$ and 
\begin{equation}
\label{spectreZ1}
d_Z(h) = d (1+ \eta h){\bf 1}\!\!  \!{\bf 1}^{*}_{[\alpha_0,\beta_0]}(h) \text{ for }h\in [0,\oo) .
\end{equation}
Moreover, $Z$ can be constructed with the additional properties:
\begin{enumerate}
\item
$\{x: h_Z(x)<+\oo \}=\{x: h_Z(x)<1 \}=\{x: h_Z(x)\leq \bbb_{0} \}$ is located on a Cantor set $\mathcal{C}$, strictly included in $\zu$,
\item
 $\zu\setminus \mathcal{C}$ consists of a countable number of open intervals whose maximal length is less than 1/10,
\item
there exists   $ 0<r_0<1$ such that  for every $x\in \zu$ and $0<r<r_0$,
\begin{equation}
\label{minorZ}
\omega_{B(x,r)} (Z)= |Z(x+r)-Z(x-r)|\leq (2r) ^{\alpha_0 }.
\end{equation} 
\end{enumerate}
\end{theorem}
%%%%%%%%%%%%%%%%%%%%%%%%%%%%%%%%%% 

\begin{definition}
\label{defFetoile}
We denote by $\caf^{*}$ the class of functions $d_{Z}$ appearing in 
\eqref{spectreZ1} with all possible choices of parameters $(\alpha_0,\beta_0,d,\eta)$ satisfying the assumptions 
of Theorem \ref{proplinearspectrum} (see Figure \ref{fig1}).
 \end{definition}

 It will be useful to keep in mind that after division by
$\aaa_{0}$ the second inequality in \eqref{*4.1*1} is of the form:
$\ds d\Big(\frac {1}{\aaa_{0}}+\hhh \Big)\leq 1$.

%%%%%%%%%%%%%%%%%%%%%%%%%%%%%%%%%%

The rest of this section is devoted to the proof of Theorem  \ref{proplinearspectrum}.

%%%%%%%%%%%%%%%%%%%%%%%%%%%%%%%%%% 
\begin{center}
\begin{figure}
  \includegraphics[width=7.0cm,height = 5.5cm]{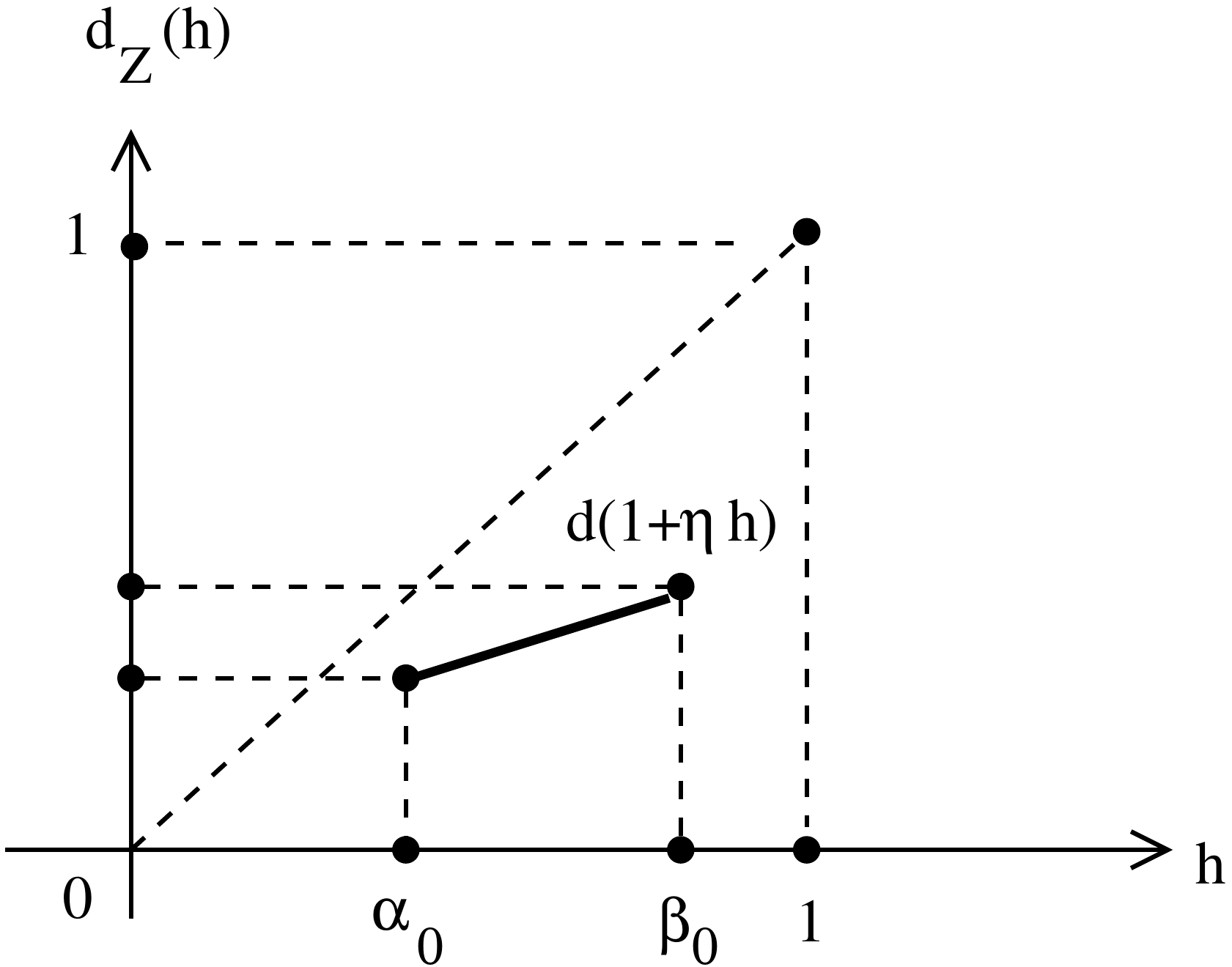}
\caption{Function in the space $\mathcal{F}^*$} \label{fig1}
\end{figure}
\end{center}
%%%%%%%%%%%%%%%%%%%%%%%%%%%%%%%%%% 

{\bf We assume $\aaa_{0}<\bbb_{0}$ and indicate 
(between parenthesis)
during the proof the places
where the case $\aaa_{0}=\bbb_{0}$ requires a slightly different argument.}

\medskip

For every integer $n\geq 1$, let us denote by 
\begin{equation*}
\alpha_{n,0}:= \alpha_0<\alpha_{n,1}< \alpha _{n,2} < \cdots <\alpha_{n,n}<\alpha_{n,n+1}:=\beta_0
\end{equation*}
the unique set of $n+2$ real numbers equally spaced in the interval $[\alpha_{0},\beta_{0}]$. 

\mk

(For the $\aaa_{0}=\bbb_{0}$ case, choose $\aaa_{0}'\in (\aaa_{0},1)$
such that 
\begin{equation}\label{*eqaa}
\frac{(1-\aaa)\aaa'}{(1-\aaa_{0})\aaa_{0}}>{1-\aaa_{0}}\text{ for }\aaa, \aaa'\in [\aaa_{0},\aaa_{0}'].
\end{equation}
For this case we set $\aaa_{n,0}=\aaa_{0}=\bbb_{0}$, $\ds \aaa_{n,1}=\aaa_{0}+\frac{\aaa_{0}'-\aaa_{0}}{n+1},$ 
$\ds \aaa_{n,2}=\aaa_{0}+\frac{2(\aaa_{0}'-\aaa_{0})}{n+1}$
and we do not define $\aaa_{n,i}$, for $i\geq 3.$)
\medskip

For every integer $n\geq 1$, we set
$$\gamma_{n,i}= d(1+\eta\alpha_{n,i})(1-10^{-n}) .$$

 By \eqref{*4.1*1} for any $\aaa\in [\aaa_{0},\bbb_{0}]$, $\ds \frac{d(1+\hhh\aaa)}{\aaa}=d\Big(\frac{1}{\aaa}+\hhh\Big)\leq
 d\Big(\frac{1}{\aaa_{0}}+\hhh\Big)\leq 1$,  hence
 \begin{equation}\label{*42b*}
 \ggg_{n,i}<\aaa_{n,i}.
 \end{equation}
 
\mk

The function $Z$  will be obtained as the sum of an infinite number of functions $Z_n$, $n\geq 1$, whose increments of order $\alpha_{n,i}$, $i\in\{1,2,\cdots, n\}$, have their cardinality controlled. Some notation for dyadic intervals are needed.
\begin{definition} 
Let $(k,j)$ be two positive integers, and $0<\alpha \leq1$ be a real number. We set $I_{j,k} = [k2^{-j},(k+1)2^{-j})$ and
\begin{equation*}
 a_{j,k}(\alpha) = (k+1)2^{-j}  -2^{-j/\alpha} \ \ \mbox{ and } \ \ I_{j,k}(\alpha) = [a_{j,k}(\alpha) ,(k+1)2^{-j}].
 \end{equation*}
Then the length of the interval $I_{j,k}(\alpha) $ is $2^{-j/\alpha}$.
\end{definition}

%%%%%%%%%%%%%%%%%%%%%%%%%%%%%%%%%%
%%%%%%%%%%%%%%%%%%%%%%%%%%%%%%%%%%
%%%%%%%%%%%%%%%%%%%%%%%%%%%%%%%%%%
%%%%%%%%%%%%%%%%%%%%%%%%%%%%%%%%%%
%%%%%%%%%%%%%%%%%%%%%%%%%%%%%%%%%%
%%%%%%%%%%%%%%%%%%%%%%%%%%%%%%%%%%
\subsection{First step}

Let us begin with the function $Z_1$. Let 
\begin{equation}\label{*eoc}
 \eee_{0}=\min\{ \aaa_{0},1-\bbb_{0} \}/2>0. 
\end{equation}

 (When $\aaa_{0}=\bbb_{0}$, we also suppose that
\begin{equation}\label{*eent}
1-\eee_{0}>\aaa_{n,2}.)\hspace{.6cm} 
\end{equation}

Consider $\alpha_{1,1}$ (which belongs to $ (\alpha_{0}, \beta_{0})$), and choose an integer $J_1$ so large that 
\begin{equation}\label{*41a11*}
2^{100}<2^{\left[ J_1 \frac{ \gamma_{1,1} }{\alpha_{1,1}} \right]+1} < 2^{J_1}/10, \  J_{1}\leq 2^{\eee_{0}J_{1}},\  
2^{-J_{1}/\bbb_{0}}<\frac{2^{-J_{1}}}{100},\text{ and }
2^{-1}\cdot 2^{J_{1}\frac{\bbb_{0}-\aaa_{0}}{2}}>1.\end{equation}

 The first inequality  can be satisfied since $\gamma_{1,1}<\alpha_{1,1}$ by \eqref{*42b*}.
 
(When $\aaa_{0}=\bbb_{0}$, we need to argue differently, since  in this case $\ds \aaa_{1,1}=\aaa_{0}+\frac{\aaa_{0}'-\aaa_{0}}{2}$ does not belong to $(\aaa_{0},\bbb_{0})$. In \eqref{*41a11*}
the last inequality should be replaced by $\ds 2^{-1}2^{J_{1}\frac{\aaa_{0}'-\aaa_{0}}{2}}>1.$) 

\mk
 
 We denote by  $\mathcal{T}_1$ the set of integers   
\begin{equation*}
\mathcal{T}_1:= \left\{k\in \{ 1,\cdots, 2^{J_1} -1\}: \ k \mbox{ is  a multiple of $2^{\left [J_1(1- \frac{\gamma_{1,1} }{\alpha_{1,1}} )\right]}$}\right\}.
\end{equation*}
 We also put $\cat_{1,1}=\cat_{1}$.
 Then $$\frac{1}{2} \cdot 2^{J_{1}-\left [J_1(1- \frac{\gamma_{1,1} }{\alpha_{1,1}} )\right]}<\# \cat_{1,1}< 2\cdot 2^{J_{1}-\left [J_1(1- \frac{\gamma_{1,1} }{\alpha_{1,1}} )\right]}.$$
 By \eqref{*41a11*} we also have $$\# \cat_{1,1}=2^{J_{1}\frac{\ggg_{1,1}}{\aaa_{1,1}}(1-\eee_{1,1})}\text{ with }\eee_{1,1}\leq \frac{1}{10}.$$

 \sk
 
 The function $Z_1$ is obtained as follows.
 \begin{itemize}
\item
for every integer $k\in \{ 0,1,\cdots, 2^{J_1} -1\}$ such that $k \not\in \mathcal{T}_1$,  for every $x\in  I_{J_1,k}$,  we set
$$Z_1(x)= x/2  .$$
Hence, $Z_1$ is just   affine on $I_{J_1,k}$, with slope $1/2$.

\sk
\item
for every integer $k\in\mathcal{T}_1$, we set for every $x\in  I_{J_1,k}$, 
$$Z_1(x)=  \begin{cases} 
2^{-1} k 2^{-J_1}  & \mbox{if }  \  x\in   [k2^{-J_1},a_{J_1,k}(\alpha_{1,1}) )
\\
 2^{-1} \left ( (k+1)2^{-J_1}  + 2^{J_1\frac{1-\alpha_{1,1}}{\alpha_{1,1}}}( x- (k+1) 2^{-J_1} ) \right )   & \mbox{if }  \ x\in  I_{J_1,k}(\alpha_{1,1}) .
  \end{cases}  $$
Hence $Z_1$ is first constant on  $[k2^{-J_1},a_{J_1,k}(\alpha_{1,1}) )$, and then affine with a large slope $2^{J_1\frac{1-\alpha_{1,1}}{\alpha_{1,1}}}$ on the interval $ I_{J_1,k}(\alpha_{1,1}) $.

\end{itemize}

%%%%%%%%%%%%%%%%%%%%%%%%%%%%%%%%%% 
\begin{center}
\begin{figure}
  \includegraphics[width=6.5cm,height = 6.0cm]{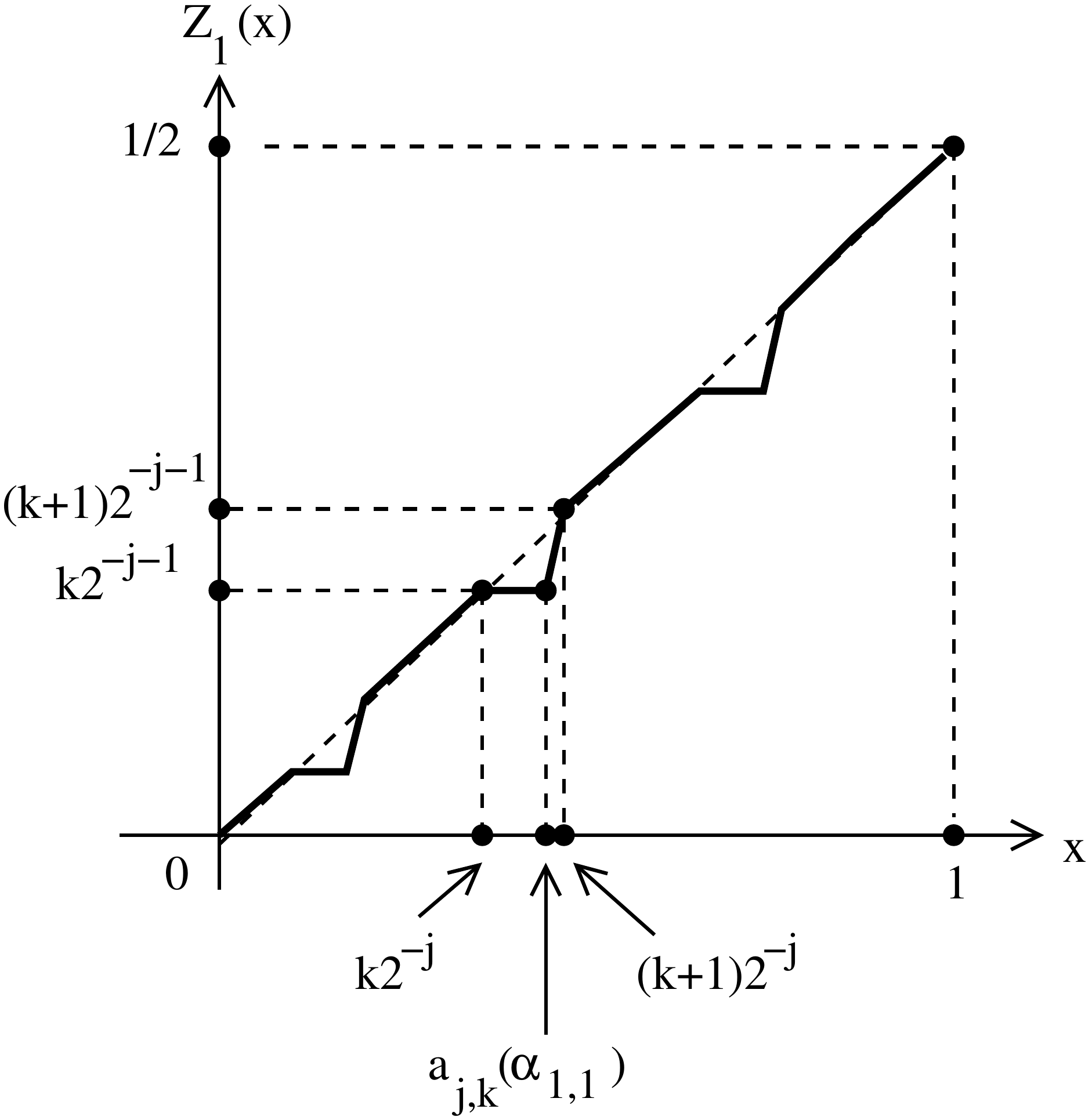}
\caption{Sketch of the graph of $Z_1$} \label{fig2}
\end{figure}
\end{center}
%%%%%%%%%%%%%%%%%%%%%%%%%%%%%%%%%% 

A quick analysis shows that the function $Z_1$ is continuous, piecewise affine, with $Z_1(0)=0$ and $Z_1(1)=1/2$, and that $Z_1(x) \leq x/2$.  Observe that the oscillations of $Z_1$ on the intervals $I_{J_1,k}(\alpha_{1,1})$, $ k \in \mathcal{T}_1$,  satisfy
$$ \om_{I_{J_1,k}(\alpha_{1,1})}(Z_1) = 2^{-1} 2^{-J_1} = 2^{-1}  |I_{J_1,k}(\alpha_{1,1})| ^{\alpha_{1,1}}.$$

We remark that the cardinality of $\mathcal{T}_1$ is $2^{J_1-{\left [J_1(1- \frac{ \gamma_{1,1}}{\alpha_{1,1}} )\right]}} \sim 2^{J_1 \frac { \gamma _{1,1} }{\alpha_{1,1}}} $.

%%%%%%%%%%%%%%%%%%%%%%%%%%%%%%%%%%
%%%%%%%%%%%%%%%%%%%%%%%%%%%%%%%%%%
%%%%%%%%%%%%%%%%%%%%%%%%%%%%%%%%%%
%%%%%%%%%%%%%%%%%%%%%%%%%%%%%%%%%%
%%%%%%%%%%%%%%%%%%%%%%%%%%%%%%%%%%
%%%%%%%%%%%%%%%%%%%%%%%%%%%%%%%%%%
\subsection{Construction of $Z_n$} 

Let $n\geq 2$, and assume that $Z_1$, $Z_2$, $\cdots$, $Z_{n-1}$ are constructed. 
We also suppose that the sets of integers
$\cat_{n-1,i}$ satisfy 
\begin{equation}\label{*482*}
\#\cat_{n-1,i}=2^{J_{n-1}\frac{\ggg_{n-1,i}}{\aaa_{n-1,i}}(1-\eee_{n-1,i})}\ ,  \ \ \text{ where } \ \eee_{n-1,i}\leq 10^{-(n-1)}.
\end{equation}

(When $\aaa_{0}=\bbb_{0}$,  we have $\cat_{n-1,i}$ only for $i=1$.
In this case in the sequel we consider only the index $i=1$, instead of $i=1,...,n.$)

Choose an integer $J_n$ satisfying the 
following conditions:
\begin{equation}\label{*4215*}
\text{
$4n\cdot 10^{n}\cdot 2^{J_{n-1} {10^{2n}}/\aaa_{0}} \leq J_n\cdot \ggg_{n,0}\leq J_{n}$  \ \ and \ \ 
$2n \cdot 2^{J_n(1-10^{-n})} \leq 2^{J_n-J_{n-1}/\alpha_0},$}
\end{equation}
moreover
\begin{equation}\label{*56a*}
4nJ_{n}\leq 2^{\eee_{0}J_{n}}, \  
2^{J_{n}(\frac{1}{\bbb_{0}}-1)}>1 \ \  \text{ and } \ \ 2^{-n}2^{J_{n}\frac{\bbb_{0}-\aaa_{0}}{n+1}}>1.
\end{equation}

(When $\aaa_{0}=\bbb_{0}$, we need to replace the last inequality
by 
\begin{equation*}
 \ds 2^{-n}2^{J_{n}\frac{\aaa_{0}'-\aaa_{0}}{n+1}}>1.)
\end{equation*}

One has
$$[J_{n}     \cdot    \ggg_{n,0}]\leq [J_{n}    \cdot   \ggg_{n,i}]\leq \Big [J_{n}\frac{\ggg_{n,i}}{\aaa_{n,i}}\Big ],$$
and by \eqref{*4.1*1}
$$\frac{\ggg_{n,i}}{\aaa_{n,i}}=d(\frac{1}{\aaa_{n,i}}+\hhh)(1-10^{-n})\leq
d(\frac{1}{\aaa_{0}}+\hhh)(1-10^{-n})\leq (1-10^{-n}).$$
Hence by \eqref{*4215*} for every $i\in\{1,\cdots,n\}$, one obtains
\begin{eqnarray*} 
      \left[J_n \frac{ \gamma_{n,i}} {\alpha_{n,i}}\right] 
\geq [J_{n}\ggg_{n,0}]	  \geq 2^{J_{n-1}/\alpha_{0}}  \  
\mbox{ and } \  \sum_{i=1}^n 2^{\left[J_n \frac{ \gamma_{n,i} }{\alpha_{n,i}}\right]} \leq 2^{J_n-J_{n-1}/\alpha_0} .
\end{eqnarray*}

Simultaneously for all exponents $\alpha_{n,i}$, $i\in\{1,2,...,n\}$, 
our aim is now to spread as uniformly as possible the intervals $I$ on which $Z_n$ has oscillations of order $|I|^{\alpha_{n,i}}$ (as we performed 
during the construction of $Z_1$).    This is achieved as follows.

\mk

Let $p_1=3$, $p_2=5$, ..., $p_n$ be the first $n$ odd  prime numbers.
 For every $i\in \{1,2,...,n\}$, we denote by  $\mathcal{T}_{n,i}$ the set of integers   
\begin{equation}
\label{defini}
\mathcal{T}_{n,i}:=\!  \left\{ \! k\in \{ 1,\cdots, 2^{J_n} \}  \! : 
  \! \begin{cases}
 \  k-p_i \mbox{ is  a multiple of $2^{\left [J_n(1- \frac{  \gamma_{n,i} }{\alpha_{n,i}} )\right]}$, and }\smallskip  \smallskip \\ 
   \mbox{ there exists an integer $0< i' \leq n-1$ such that } \sk\sk\\
 \ k 2^{-J_n}  \mbox{  belongs to } I_{J_{n-1},K} ( \alpha_{n,i'} )  \mbox{ for some $K\in  \mathcal{T}_{n-1,i'} $}
 \end{cases} 
 \!\!\! \!\!  \!\!  \right\} \!.
\end{equation}

(When $\aaa_{0}=\bbb_{0}$, we use only $p_{1}=3$, consider
only $\cat_{n,1}$  and use only $i'=1$
in the definition of $\cat_{n,i}=\cat_{n,1}$.)

To estimate the cardinality of $\mathcal{T}_{n,i}$ we have 
\begin{eqnarray*}  & &
 \frac{1}{2}\sum_{i'=1}^{n-1}  \left(\#\mathcal{T}_{n-1,i'}  \right)   2^{-J_{n-1}/ \alpha_{n,i'}}  \cdot   2^{ J_{n }-{\left [J_n(1- \frac{ \gamma_{n,i} }{\alpha_{{n,i}}} )\right]}}
 \\
&&< \ 
\# \mathcal{T}_{n,i} 
  \ <  \ 2\cdot \sum_{i'=1}^{n-1}  \left(\#\mathcal{T}_{n-1,i'}  \right)   2^{-J_{n-1}/ \alpha_{n,i'}}  \cdot   2^{ J_{n }-{\left [J_n(1- \frac{ \gamma_{n,i} }{\alpha_{{n,i}}} )\right]}} \\
& &
<4 n   2^{ J_{n } \frac{ \gamma_{n,i} }{\alpha_{{n,i}}}},
\end{eqnarray*}
that is,
$$ \#\mathcal{T}_{n,i} \sim \   2\cdot  \sum_{i'=1}^{n-1}   \left(\#\mathcal{T}_{n-1,i'}  \right)  2^{-J_{n-1}/ \alpha_{n,i'}}  \cdot    2^{J_n \frac {\gamma_{n,i}  }{\alpha_{n,i}}}.$$
Hence, using \eqref{*482*} and \eqref{*4215*} we have 
\begin{equation}
\label{majtni}
\# \mathcal{T}_{n,i} =   2^{J_n \frac {\gamma_{n,i}  }{\alpha_{n,i}} (1-\ep_{n,i})} \ \text{ where  \ $\ep_{n,i} \leq 10^{-n}$.}  
\end{equation}
    Finally, we set
 \begin{equation}
 \label{defI}
 \calt_n \ = \ \bigcup_{i=1}^n  \ \calt_{n,i}.
\end{equation}

(When $\aaa_{0}=\bbb_{0}$, we have $\cat_{n}=\cat_{n,1}.$)
 %%%%%%%%%%%%%%%%%%%%%
\begin{lemma}
If $ 1 \leq i<j \leq n$, then $\mathcal{T}_{n,i} \cap \mathcal{T}_{n,j} =\emptyset$.
Moreover, if $k \in \calt_{n,i}$ and  $k' \in \calt_{n,j}$ such that $k\neq k'$ (one may have $i=j$), then $I_{J_n,k}(\alpha_{n,i})  \bigcap I_{J_n,k'}(\alpha_{n,j})  =\emptyset$.
 \end{lemma}
  
%%%%%%%%%%%%%%%%%%%%%
 
%%%%%%%%%%%%%%%%%%%%%
 \begin{proof}
Suppose that an integer $q$ belongs to $\mathcal{T}_{n,i} \cap \mathcal{T}_{n,j} $. Hence $q$ can be written
 $$ q= p_i+ m_i 2^{\left [J_n(1- \frac{ \gamma_{n,i}  }{\alpha_{n,i}} )\right]} = p_j+ m_j 2^{\left [J_n(1- \frac{ \gamma_{n,j}}{\alpha_{n,j}} )\right]}.$$
 Assume that $2^{\left [J_n(1- \frac{\gamma_{n,i}  }{\alpha_{n,i}} )\right]} \geq 2^{\left [J_n(1- \frac{ \gamma_{n,j}  }{\alpha_{n,j}} )\right]}$, the other case is similar. Then
 \begin{equation}
 \label{eq1}
 0\not=|p_j-p_i| = 2^{\left [J_n(1- \frac{ \gamma_{n,j}  }{\alpha_{n,j}} )\right]}\left|  m_j -  m_i2^{\left [J_n(1- \frac{  \gamma_{n,i} }{\alpha_{n,i}} )\right]- \left [J_n(1- \frac{ \gamma_{n,j}  }{\alpha_{n,j}} )\right]}\right|
> 2^{\left [J_n(1- \frac{ \gamma_{n,j}  }{\alpha_{n,j}} )\right]}. \end{equation}
 
 Obviously, 
 by Bertrand's postulate (proved by Chebyshev) about prime numbers
we have 
 $0<|p_j-p_i| \leq p_n \leq 2^{n+1}$, while $2^{\left [J_n(1- \frac{ \gamma_{n,j}  }{\alpha_{n,j}} )\right]}   >\!\!>2^{n+1}$. Hence,  it is impossible to realize \eqref{eq1}.
 
Finally,  if $k \neq k'$,  since $\aaa_{n,i}$ and $\aaa_{n,j}$ are   smaller
than one,   it is clear from the construction that $I_{J_{n},k}(\aaa_{n,i})\cap I_{J_{n},k'}(\aaa_{n,j})=\ess.$
 \end{proof}
 %%%%%%%%%%%%%%%%%%%%%

 The non-decreasing mapping $Z_n$ is obtained as follows.
 \begin{itemize}
\item
for every integer $k\in \{ 0,1,\cdots, 2^{J_n} -1\}$ such that $k \not\in \calt_n$,  for every $x\in  I_{J_n,k}$, we set: 
$$Z_n(x)=     2^{-n} x.$$
 
 \item
When the integer $k$ belongs to some $\mathcal{T}_{n,i}$, we set:   for every $x\in  I_{J_n,k}$, 
$$Z_n(x)=  \begin{cases} 
2^{-n} k 2^{-J_n}   & \mbox{if }  \  x\in [k2^{-J_n},a_{J_n,k}(\alpha_{n,i})),\\
2^{-n}  \Big( (k+1)2^{-J_n}   +  2^{J_n\frac{1-\alpha_{n,i}}{\alpha_{n,i}}}( x- (k+1) 2^{-J_n} ) \Big) & \mbox{if }  \ x\in  I_{J_n,k}(\alpha_{n,i}) .
 \end{cases}  $$

\end{itemize}

As above, the function $Z_n$ is continuous, piecewise affine, and it obviously satisfies $Z_n(0)=0$, $Z_n(1)=2^{-n}$ and $Z_n(x)\leq 2^{-n}$.  Moreover,   the oscillations of $Z_n$ on the intervals $I_{J_n,k}(\alpha_{n,i})$, $ k \in \mathcal{T}_{n,i}$,  satisfy
$$ \om_{I_{J_{n},k}(\alpha_{n,i})}(Z_n) = 2^{-n} 2^{-J_n} = 2^{-n}  |I_{J_n,k}(\alpha_{n,i})| ^{\alpha_{n,i}}.$$

In addition, a key remark is that $Z_n$ is not  a linear function with slope $2^{-n}$ only on some intervals $I_{J_n,k}$ that are included in intervals $I_{J_{n-1},k'}$ on which $Z_{n-1}$ has large oscillation. Hence, these nested intervals are the intermediary steps of the construction of a Cantor set.

%%%%%%%%%%%%%%%%%%%%%%%%%%%%%%%%%%
%%%%%%%%%%%%%%%%%%%%%%%%%%%%%%%%%%
%%%%%%%%%%%%%%%%%%%%%%%%%%%%%%%%%%
%%%%%%%%%%%%%%%%%%%%%%%%%%%%%%%%%%
%%%%%%%%%%%%%%%%%%%%%%%%%%%%%%%%%%
%%%%%%%%%%%%%%%%%%%%%%%%%%%%%%%%%%
\subsection{Construction  of $Z$} 

\begin{definition}
We define the mapping $Z:\zu\to \zu$ by the formula
$$\mbox{for every $x\in\zu$, } \ \ Z(x) = \sum_{n=1}^{+\infty} \ Z_n(x).$$
\end{definition}

Immediate properties of $Z$ are gathered in the next Proposition.
\begin{proposition}
The mapping $Z$ is continuous (as uniform limit of continuous functions), strictly increasing, and satisfies $Z(0)=0$ and $Z(1)=1$.
\end{proposition}

This obviously follows from the construction, and from the fact that 
$$ \mbox{for every $N\geq 1$, } \  \ \|Z-Z_N\|_{\infty} = \Big  \| \sum_{n\geq N+1} Z_n \,\Big\|_{\infty} \leq  2^{-N+2}.$$

 Moreover, let us call $\mathcal{C}$ the (closed) Cantor set
 \begin{equation}
 \label{defcantor}
 \mathcal{C} = \bigcap_{N\geq 1} \ \bigcup_{n\geq N}  \ \bigcup_{i=1}^{n}  \ \ \bigcup_{k\in\!\!\displaystyle{\bigstar_{n}}} \ I_{J_{n},k}(\alpha_{n,i}),
 \end{equation} 
$\bigstar_{n} $ meaning that   the intervals $I_{J_n,k}(\alpha_{n,i})$ considered are only for those $k$ which appear in the construction of $Z_n$. 

\begin{lemma}
\label{lem00}
If $x\notin \mathcal{C}$, then $h_Z(x)=+\infty$.
\end{lemma}
\begin{proof}
If $x\notin \mathcal{C}$, then there exists $r>0$ such that the distance between $x$ and $\mathcal{C}$ is larger than $r$. In particular, $x$ does not belong to any interval of the form $I_{J_{n},k}(\alpha_{n,i})$ appearing in the definition of $\mathcal{C}$, for $n$ larger than some integer $N>0$. In other words, $Z$ is affine in a neighborhood of $x$, hence the conclusion.
\end{proof}

Lemma \ref{lem00} and our construction  yield  items (i) and (ii) of Theorem \ref{proplinearspectrum} concerning the size of the complement of the Cantor set.

%%%%%%%%%%%%%%%%%%%%%%%%%%%%%%%%%%
%%%%%%%%%%%%%%%%%%%%%%%%%%%%%%%%%%
%%%%%%%%%%%%%%%%%%%%%%%%%%%%%%%%%%
%%%%%%%%%%%%%%%%%%%%%%%%%%%%%%%%%%
%%%%%%%%%%%%%%%%%%%%%%%%%%%%%%%%%%
%%%%%%%%%%%%%%%%%%%%%%%%%%%%%%%%%%
\subsection{Local regularity properties of $Z$}

To find the  multifractal  spectrum of $Z$, we start by studying its local oscillations.

For every integers $N$, $k$, $i$ and every   $r>0$, let us define the intervals  $I_{J_N,k}(i,r)$  by:
\begin{equation*}
I_{J_N,k}(i,r)  =   I_{J_N,k}(\alpha_{N,i })+ B(0,r).
\end{equation*}

\begin{definition}
For $i_{0}=1,...,N+1$ let us introduce the sets
\begin{eqnarray*}
 \mathcal E_{N,i_0,r} & =& \{ x \in \zu: \om_{B(x,r )}(Z)\geq (2 r)^{ \alpha_{N,i_0} } \},\\
\mathcal E'_{N,i_{0}} & = &   \bigcup _{i:i<i_{0}}  \ \  \bigcup _{k\in\cat_{N,i}} I_{J_N,k}(\aaa_{N,i}), \\
\mathcal  E''_{N,i_{0},r} & = &   \bigcup _{i:i<i_{0}} \ \   \bigcup _{k\in\cat_{N,i}} I_{J_N,k}(i,r) .
 \end{eqnarray*}
\end{definition}

(When $\aaa_{0}=\bbb_{0}$, we consider these sets only for $i_{0}=1,2$,
this restriction applies in Lemma \ref{lemmadecompos}  as well.)

Heuristically, $\mathcal E'_{N,i_{0}} $ contains all the intervals on which $Z_N$ has an exponent less than $\alpha_{N,i_0}$, and $\mathcal E''_{N,i_{0},r}$ is the $r$-neighborhood of these points. Also, $\mathcal E''_{N,1,r}=\ess.$ 

We also remark that $\ds \cac =\bigcap_{N=1}^{\oo}\mathcal E'_{N,N+1}$.
(When $\aaa_{0}=\bbb_{0}$ then $\ds \cac =\bigcap_{N=1}^{\oo}\mathcal E'_{N,2}$.)

The next  Lemma,  very technical,  allows us to "locate" the elements around which $Z$ has an oscillation of a given size.

%%%%%%%%%%%%%%%%%%%%%%%%%%%%%%%%%%
%%%%%%%%%%%%%%%%%%%%%%%%%%%%%%%%%%
\begin{lemma}
\label{lemmadecompos}
Let $r\in (0,2^{-J_{5}-1})$,  and let $N\geq 5$ be the unique integer such that 
\begin{equation}
\label{defJN}
2^{-J_{N+1}} \leq  2r <2^{-J_N}.
\end{equation}

Let $i_{0}\in \{ 1,...,N+1 \}$. One has:
\begin{enumerate}
\smallskip
\item
If  \ $2^{-\frac{J_{N}}{\aaa_{0}(1-\bbb_{0})}}<2r<2^{- \frac{J_{N}}{\aaa_{N,i_{0}}}}$, 
then $\mathcal E_{N,i_{0},r} \sse \mathcal E''_{N,i_{0},r}$.
\smallskip
\item
 If \ $2r>2^{- \frac{J_{N}}{\aaa_{N,i_{0}}
}}$,
or  if  \ $2r<2^{-\frac{J_{N}}{\aaa_{0}(1-\bbb_{0})}}$, then $\mathcal E_{N,i_{0},r}  =\ess$.
\smallskip
\item
Moreover,  if $x\in \mathcal E'_{N,i_{0}}$, then there is
$0<r\leq 2^{-J_{N}/\aaa_{N,i_{0}}}$  such that $x\in \mathcal E_{N,i_{0},r} .$
\end{enumerate}
 \end{lemma}
%%%%%%%%%%%%%%%%%%%%%%%%%%%%%%%%%%
%%%%%%%%%%%%%%%%%%%%%%%%%%%%%%%%%%

%%%%%%%%%%%%%%%%%%%%%%%%%%%%%%%%%%
\begin{proof}

 We are going to investigate the possible values of $\om_{B(x,r)}(Z)$, for all possible  values of $x\in\zu$ and $r>0$. 
We also emphasize that 
$0<\ep_{0}$  by \eqref{*eoc} is so small that $0<\alpha_0-\ep_{0} < \beta_0 +\ep_{0} <1$. 
 Recall also that we supposed that $N$ is the only integer satisfying \eqref{defJN}.

 We fix $i_0\in \{1,..., N+1\}$, and we look for the locations of the elements of $\mathcal E_{N,i_{0},r}$.

 Obviously, $\om_{B(x,r)}(Z)  = \sum_{n\geq 1} \om_{B(x,r)}(Z_n)$. 
Let us compare the terms in  this sum  according to  the value of $n$. 

\sk\sk\sk

{\bf (i)  ${\bf  n \geq N+1}$:} by construction, $B(x,r)$ is covered by at most $(2r  /2^{-J_n}) +2 \leq 4r2^{ J_n}$ 
dyadic intervals of generation $J_n$, on which the oscillation of $Z_n$ is exactly $2^{-J_n} 2^{-n}$. 
Hence     $ \om_{B(x,r)}(Z_n) \leq 4r2^{J_n} 2^{-J_n} 2^{-n} \leq 4 \cdot 2^{-n} r$. Since $N> 4$ summing over $n\geq N+1$ yields
\begin{equation}
\label{eq11}
\sum_{n\geq N+1} \om_{B(x,r)}(Z_n) \leq 8\cdot 2^{-N} r\leq   (2r)/4 \leq (2r)^{\alpha_{N,i_0}}/4.
\end{equation}

\sk\sk\sk

{\bf (ii)   ${\bf n\leq N-1}$:} by construction,  $Z_n$ has a 
maximal
slope of $2^{-n}2^{J_n\frac{1-\alpha_{n,j}}{\alpha_{n,j}}}$, for some integer $j\in\{1,\cdots, n\}$, which by $2^{-n}<1$, $1-\aaa_{n,j}<1$, $\aaa_{n,0}<\aaa_{n,j}$ verifies
\begin{equation}\label{*55b*}
 2^{-n}2^{J_n\frac{1-\alpha_{n,j}}{\alpha_{n,j}}} < 2^{J_n\frac{1 }{\alpha_{n,0}}} =2^{J_n/{\alpha_{0}}}   .\end{equation}
By \eqref{*4215*} and \eqref{*56a*}, for every $n<N$ one has
\begin{equation}\label{*51bis} 
\text{
$J_N \geq 2^{J_{n }  /\alpha_0}$ \ \ and  \ \  $    4N J_N \leq 2^{\ep_{0} J_N   }$.}
\end{equation}
  By \eqref{*55b*} we have   $ \om_{B(x,r)}(Z_n) \leq 2^{J_n /\alpha_0}  2r $.
By \eqref{defJN} and \eqref{*51bis} we obtain
 $$
\sum_{n=1}^{N-1} \om_{B(x,r)}(Z_n) \leq  \sum_{n=1}^{N-1}  2^{J_n /\alpha_0} (2r)  \leq   4 N J_N  (2r)/4\leq 2^{\ep_{0} J_N   } 2r /4 \leq   (2 r)^{1- \ep_{0}}/4.
$$
Since $1-\ep_{0} > \beta_0 = \alpha_{N,N+1}\geq \alpha_{N,i_0}$, we deduce that 
\begin{equation}
\label{eq12}
\sum_{n=1}^{N-1} \om_{B(x,r)}(Z_n) \leq (2r)^{\alpha_{N,i_0}}/4.
\end{equation}

(When $\aaa_{0}=\bbb_{0}$, we also use \eqref{*eent}.)
\sk\sk\sk

{\bf (iii)   ${\bf n=N}$:} It remains  to study the oscillation of $Z_N$ on $B(x,r)$.   
 By \eqref{eq11}  and \eqref{eq12}, if $x\in \mathcal E_{N,i_{0},r}$, then one necessarily has 
\begin{equation}
\label{ineg00}
\om_{B(x,r )}(Z_N)\geq (2r)^{  \alpha_{N,i_0} }/2.
\end{equation}
Our goal is now to identify the elements satisfying  \eqref{ineg00}.

 By \eqref{defJN}, $B(x,r)$ contains at most one dyadic number $k2^{-J_N}$.  Recall that  $\calt_N$ defined by \eqref{defI} contains the integers $k$ such that  $Z_N$ possesses large oscillations on  intervals    of the form $I_{J_N,k} (\alpha_{N,i})$. If $B(x,r)$ does not meet any interval of the form $I_{J_N,k}(\alpha_{N,i})$, then the oscillation of $Z_N$ on $B(x,r)$ is less than $2^{-N} 2r  $ (since $2^{-N}$ is the value of the slope of $Z_N$ on such intervals), and \eqref{ineg00} cannot be realized.

 \sk
 
 We thus assume that $B(x,r)$  intersects   an interval of the form $I_{J_N,k}(\alpha_{N,i})$, where 
 $i \in \{0,1,...,N\}$ and  $k\in \calt_{N,i}$.   Observe that $B(x,r)$ can intersect at most two such intervals, and when it does, $2r \sim 2^{-J_N}$ and $\om_{B(x,r)}(Z_N)\leq 2\cdot 2^{-N} \cdot 2^{-J_N}$. 
Thus, in this case, $x\not\in \mathcal E_{N,i_{0},r}$.
 
 \mk
 
 We thus assume that  $B(x,r)$  intersects exactly  one interval of the form $I_{J_N,k}(\alpha_{N,i})$, where 
 $i \in \{0,1,...,N\}$ and  $k\in \calt_{N,i}$.
 
 \sk \sk \sk
 
 {\bf 1.} {\bf  If $\mathbf{2r >   2^{-J_N /   \alpha_{N,i_0}  }}$:}  then  by definition of $Z_N$ 
 the oscillation  on the intersection $I_{J_N,k}(\alpha_{N,i}) \cap B(x,r)$ is  less than the oscillation of $Z_N$ on one interval of size $2^{-J_N}$, i.e. less than 
 $2^{-N}2^{-J_N} \leq  2^{-N} (2r) ^{ \alpha_{N,i_0}  }   \leq  (2r) ^{ \alpha_{N,i_0}  } /4$. 
 Hence,  $x$ cannot belong to $\mathcal  E_{N,i_{0},r}$.

 \mk \sk
 {\bf 2.} {\bf  Next we assume:}
 \begin{equation}
 \label{conditiononr}
 \mathbf {2r \leq   2^{-J_N /  \alpha_{N,i_0}  }.}
 \end{equation}
 
 \sk \sk

\noindent {\bf 2A.   If  ${\bf 2r< 2^{\frac{-J_{N}}{(1-\bbb_{0})\aaa_{0}}}}$:} 
Keeping in mind that $1>1-\aaa_{N,i}$, $1-\bbb_{0}\leq 1-\aaa_{N,i_{0}}$
and $\aaa_{0}\leq \aaa_{N,i}$
this implies that   
\begin{equation}\label{*raa}
2r < 2^{\frac{-J_{N}}{(1-\bbb_{0})\aaa_{0}}}<2^{\frac{-J_{N}(1-\aaa_{N,i})}{(1-\aaa_{N,i_{0}})\aaa_{N,i}}}.
\end{equation}

(At this point, when $\aaa_{0}=\bbb_{0}$,   extra care is needed.
In this case $i_{0}=1$, or $2$. To obtain \eqref{*raa} one can use \eqref{*eqaa}
with $\aaa=\aaa_{N,i_{0}}$ and $\aaa'=\aaa_{N,i}$ to obtain
$$\frac{(1-\aaa_{N,i_{0}})\aaa_{N,i}}{(1-\aaa_{0})\aaa_{0}}>{1-\aaa_{0}}>{1-\aaa_{N,i}}.\hspace{.6cm})$$

 By \eqref{*raa}, one gets
 $$\ooo_{B(x,r)}(Z_{N})< 2^{-N} 
 2^{J_{N}\frac{1-\aaa_{N,i}}{\aaa_{N,i}}}2r< 2^{-N}(2r)^{-1+\aaa_{N,i_{0}}}(2r)\leq (2r)^{\aaa_{N,i_{0}}}/4,$$
 and $x\not\in \mathcal  E_{N,i_{0},r}$. Since this argument applies for any
 $x\in [0,1]$ for ${2r< 2^{\frac{-J_{N}}{(1-\bbb_{0})\aaa_{0}}}}$
 we have
  $\mathcal E_{N,i_{0},r} =\ess$.

\mk\sk

\noindent {\bf 2B. If ${\bf  2^{\frac{-J_{N}}{(1-\bbb_{0})\aaa_{0}}}<2r <2^{-J_N /  \alpha_{N,i_0}  } }$ and   ${\bf \mathbf{\alpha_{N,i} \geq\alpha_{N,i_0} }}$:}   the slope of the function $Z_N$ on the interval $I_{J_N,k}(\alpha_{N,i})$  is exactly $2^{-N}2^{J_N\frac{1- \alpha_{N,i}}{\alpha_{N,i}}}$.
The ball $B(x,r)$ can also intersect $I_{J_{N},k-1}$, or $I_{J_{N},k+1}$
but by \eqref{*56a*} on these intervals the slope of $Z_{N}$ equals
$2^{-N}< 2^{-N}2^{J_N\frac{1- \alpha_{N,i}}{\alpha_{N,i}}}$.
Consequently,  the oscillation of $Z_N$ on $ B(x,r)$  is less than 
 $2^{-N}2^{J_N\frac{1- \alpha_{N,i}}{\alpha_{N,i}}} 2 r $, which by  \eqref{conditiononr} and ${\alpha_{N,i_0}} /{\alpha_{N,i}}\leq 1$  is no more than
$$   2^{-N}  (2r)^{-\alpha_{N,i_0} \frac{1- \alpha_{N,i}}{\alpha_{N,i}} } (2r)  \leq   (2r)^{({ \alpha_{N,i} -1})  \frac{\alpha_{N,i_0}} {\alpha_{N,i}} +1} /4 \leq    (2r) ^{  \alpha_{N,i}}/4  \leq (2r)^{\alpha_{N,i_0} } /4.$$
Once again, $x\notin \mathcal  E_{N,i_{0},r} $.

 \sk \sk \sk

\noindent {\bf 2C.  If ${\bf  2^{\frac{-J_{N}}{(1-\bbb_{0})\aaa_{0}}}<2r <2^{-J_N /  \alpha_{N,i_0}  } }$ and   ${\bf \ \alpha_{N,i} < \alpha_{N,i_0} }$:}    observe   that if $x\not\in \mathcal E''_{N,i_{0},r}$, then $B(x,r)$ does not intersect
 any interval of the form $I_{J_N,k}(\alpha_{N,i })$
 with $\aaa_{N,i}\leq \aaa_{N,i_{0}}$. Hence,    $x\not \in \mathcal E_{N,i_{0},r} $. We deduce that necessarily $\mathcal E_{N,i_{0},r}  \subset \mathcal E''_{N,i_{0},r}$.

 \mk 
 
 This concludes the proof of the  parts (i) and (ii)  of Lemma \ref{lemmadecompos}.
  
 \medskip
   
To obtain part (iii) of the lemma, assume that  $x\in \mathcal E'_{N,i_{0}}$, for some $i_0\in\{1,..., N+1\}$. 

(When $\aaa_{0}=\bbb_{0}$,  this means $i_{0}=2$, since
$\mathcal E'_{N,1}=\ess$.)

Then there exists $i<i_{0}$   
and  $k\in \cat_{N,i}$  such that  $x\in I_{J_N,k}(\alpha_{N,i })$.
Choose $r=2^{-J_{N}/\aaa_{N,i}}< 2^{-J_{N}/\aaa_{N,i_{0}}}$.
Then $B(x,r)\supset I_{J_N,k}(\alpha_{N,i })$ and
$$\ooo_{B(x,r)}(Z_{N})\geq 2^{-N} \cdot 2^{-J_{N}}=
2^{-{N}}(2^{-J_{N}/\aaa_{N,i}})^{\aaa_{N,i_{0}}}(2^{-J_{N}/\aaa_{N,i}})^{\aaa_{N,i}-\aaa_{N,i_{0}}}.$$
Using \eqref{*56a*}, one concludes that
$$\ooo_{B(x,r)}(Z_{N})\geq2^{-N}2^{J_{N}\frac{\aaa_{N,i_{0}}-\aaa_{N,i}}{\aaa_{N,i}}}
\cdot r^{\aaa_{N,i_{0}}}\geq 2^{-N}2^{J_{N}\frac{\bbb_{0}-\aaa_{0}}{N+1}}r^{\aaa_{N,i_{0}}}>r^{\aaa_{N,i_{0}}}.$$
Consequently, $x\in \mathcal  E_{N,i_0,r} $.
  \end{proof}
 %%%%%%%%%%%%%%%%%%%%%%%%%%%%%%%%%%
%%%%%%%%%%%%%%%%%%%%%%%%%%%%%%%%%%

From these considerations, one deduces easily item (iii) of Theorem \ref{proplinearspectrum}: for every $x\in \zu$, for every $r>0$ small enough,  $Z(x+r)-Z(x-r) \leq  2r^{\alpha_0}$. 

Indeed, consider  $0<r<2^{-J_{5}-1}$ and the associated integer $N> 4$  such that 
\eqref{defJN} holds. By Lemma \ref{lemmadecompos} we have
$\mathcal E_{N,1,r}\sse \mathcal E''_{N,1,r}=\ess$.
Hence, for any $x\in[0,1]$,
\begin{equation}\label{*19201*}
\om_{B(x,r)}(Z)=Z(x+r)-Z(x-r)<(2r)^{\aaa_{N,1}}<(2r)^{\aaa_{N,0}}=
(2r)^{\aaa_{0}}.
\end{equation}
  
 \mk

  %%%%%%%%%%%%%%%%%%%%%%%%%%%%%%%%%%
%%%%%%%%%%%%%%%%%%%%%%%%%%%%%%%%%%
\subsection{Upper bound for the spectrum of $Z$} We need to consider only points within the Cantor set $\mathcal{C}$.

   \medskip

 Let $\alpha_0\leq \alpha\leq \beta_0$, and assume that a real number $x$ satisfies $ h_Z(x)\leq \alpha$.
(When $\aaa_{0}=\bbb_{0}$,  we need to consider only the case
$\aaa=\aaa_{0}$.)
 By definition, for every $\ep>0$ such that $0 <\alpha_0-\ep<\beta_0+\ep<1$,  one can find a strictly increasing sequence of integers $p$ such that $x$  belongs to an  $\mathcal E_{N_p,i_{N_p },2^{-p}} (Z )$, where $i_{n }$ is the largest  integer satisfying $\alpha_{n,i_{n }} \leq \alpha+\ep$, ($i_n$ depends on $\alpha$ and $\ep$, but we omit the subscripts for clarity). Otherwise we would have for every small $r>0$
 $$\om_{B(x,r)} < (2r)^{\alpha +\frac{\eee}{2}},$$
 which contradicts the fact that $ h_Z(x)\leq \alpha$.

  Hence, we have the inclusion 
\begin{equation}
\label{upperbound}
E_Z^{\leq} (\alpha)= \{x:h_Z(x)\leq \alpha\} \subset \limsup_{p\to +\infty} \mathcal  E_{N_{p},i_{N_p},2^{-p}}   =  \bigcap_{P\geq 1} \ \bigcup_{p\geq P}   \ \mathcal E_{N_p,i_{N_p},2^{-p}}  .
\end{equation}
 Using this, we prove that the Hausdorff dimension of $E_Z^{ \leq}(\alpha) $ is less than $d(1+\eta\alpha)$. 
 Let $s> d(1+\eta\alpha)$. 
 
 Obviously, by \eqref{upperbound}, $E_Z^{\leq} (\alpha)$ is covered for every   $P\geq 1 $ by the union  
 $\bigcup_{p\geq P}   \ \mathcal E_{N_p,i_{N_p},2^{-p}}  .$
 Let us count the number of intervals in  
$\mathcal E_{N_p,i_{N_p },2^{-p}} $. 
 Recalling  \eqref{majtni} we deduce using  the decomposition above that  $\mathcal E_{N_p,i_{N_p},2^{-p}} $   is covered by the union over each $i$ 
 such that $i< i_{N_p}$  of  
  $$2\cdot 2^{J_{N_p}   \frac {\gamma_{N_p,i }}{\alpha_{N_p,i }}(1-\ep_{N_p,i})}   $$
   intervals of the form $ I_{J_{N_p}  ,k}(i,2^{-p})  $  (here the value of $i$  
   depends on the value of $2^{-p}$).
   But as noticed above, in order to have $h_Z(x)\leq \alpha$,  one necessarily has $x\in  I_{J_{N_p}  ,k}(\alpha_{N_p,i})$  (not only $x \in I_{J_{N_p}  ,k}(i,2^{-p})  $). 
   
 Hence, the $s$-Hausdorff pre-measure $\mathcal{H}^s_\delta$ (which is obtained by using covering of size less than $\delta\geq 2^{-p}$) of  $\mathcal E_{N_p,i_{N_p},2^{-p}} $ is bounded from above by
 \begin{eqnarray*}
 \mathcal{H}^s_\delta( \mathcal E_{N_p,i_{N_p},2^{-p}}) & \leq & \sum_{i=1}^{i_{N_p}-1}  |  I_{J_{N_p}  ,k}(\alpha_{N_p,i })|^s  \cdot 2\cdot 2^{J_{N_p}   \frac {\gamma_{N_p,i }}{\alpha_{N_p,i  }}(1-\ep_{N_p,i})} \\
  & \leq & C    \sum_{i=1}^{i_{N_p}-1} 
  2^{J_{N_p}  \left(   \frac {\gamma_{N_p,i }}{\alpha_{N_p,i  }}(1-\ep_{N_p,i})- \frac{s}{\alpha_{N_p,i}} \right)} .
  \end{eqnarray*}
Then, 
 \begin{eqnarray*}
 \mathcal{H}^s_\delta(  E^{\leq}_Z(\alpha)) & \leq & \sum_{p\geq P }  \mathcal{H}^s_\delta( \mathcal E_{N_p,i_{N_p},2^{-p}}) \\
 & \leq & C  \sum_{p\geq P } \sum_{i=1}^{i_{N_p}-1}   2^{J_{N_p}  \left(   \frac {\gamma_{N_p,i }}{\alpha_{N_p,i  }}(1-\ep_{N_p,i})- \frac{s}{\alpha_{N_p,i}} \right)} .
  \end{eqnarray*}
This series converges since $s >d(1+\eta \alpha) >\gamma_{N_p,i_{N_p}} >\ggg_{N_{p},i}$ and $\ep_{N_p,i} \leq 10^{-N_p}$ (for every indices $p$ and $i$).
We conclude that  $ \mathcal{H}^s_\delta( E^{\leq}_Z(\alpha) ) =0$,  hence $\dim_{\mathcal{H}} E^{\leq}_Z(\alpha) \leq s$. Since this holds for every $s>d(1+\eta \alpha)$ we obtained our upper bound for the spectrum of $Z$  
when $\aaa\in [\aaa_{0},\bbb_{0}]$.

By \eqref{minorZ}, $h_{Z}(x)\geq \aaa_{0}$ for any $x\in [0,1]$
and hence $\{ x:h_{Z}(x)<\aaa_{0} \}=\ess.$

On the other hand, in order to have $h_{Z}(x)<+\oo$
one needs $x\in \cap_{N=1}^{\oo}  \mathcal E'_{N,N+1}$.
But then by Lemma \ref{lemmadecompos} there exists $r_{N}\to 0$
 such that $x\in\mathcal  E_{N,N+1,r_{N}} $. Since $\aaa_{N,N+1}= \bbb_{0}$ we obtain that in this case $h_{Z}(x)\leq \bbb_{0}.$
 
 (When $\aaa_{0}=\bbb_{0}$,  we have $x\in \mathcal E_{N,2,r_{N}}$ and
 $\aaa_{N,2}\to\aaa_{0}=\bbb_{0}.$)

%%%%%%%%%%%%%%%%%%%%%%%%%%%%%%%%%%
%%%%%%%%%%%%%%%%%%%%%%%%%%%%%%%%%%
%%%%%%%%%%%%%%%%%%%%%%%%%%%%%%%%%%
%%%%%%%%%%%%%%%%%%%%%%%%%%%%%%%%%%
\subsection{Lower bound for the spectrum of $Z$}
  Let $\alpha \in [\alpha_0,\beta_0]$, and consider  the sequence of sets
 $\mathcal{F}_{n} (\alpha)= \bigcup_{k\in \mathcal{T}_{n,i _{n}}} I_{J_{n},k} (\alpha_{n,i_{n}})$,
 where $i_{n}\in \{ 1,...,n \}$ is the largest  integer satisfying $\alpha_{n,i_{n}} < \alpha$.
(In case of $\aaa=\aaa_{0}$, such an integer does not exist
and we set $i_{n}=1$ in this case.)

 Consider the Cantor set 
  $$\mathcal{F}(\alpha) = \bigcap_{n\geq 2} \mathcal{F}_{n} (\alpha).$$
Obviously $\mathcal{F}(\alpha) \subset \mathcal{C}=\cap_{N=1}^{\oo} E'_{N,N+1}$, where  $\mathcal{C}$ is the Cantor set  defined by \eqref{defcantor} on which the function $Z$ has exponents between $\alpha_0$ and $\beta_0$.

First remark  that by \eqref{*4215*} the sequence $(J_n)_{n\geq 2}$ is lacunary (for instance, we have $J_n \geq 2^{J_{n-1}}$).
In addition, recall that by \eqref {majtni}, one has
$$\# \mathcal{T}_{n,i_{n}} =   2^{J_{n} \frac {\gamma_{n,i_{n}}  }{\alpha_{n,i_{n}}}  (1-\ep_{n,i_{n}})}  .$$
From the lacunarity of $(J_n)$, one deduces that
$$  \log  (\prod_{m=1}^n \ \# \mathcal{T}_{m,i_{m}}) \sim_{n\to +\infty}  \log   \ \# \mathcal{T}_{n,i_{n}}.$$
Finally, the intervals of $I$, which are all of length  $ 2^{-J_{n} /{\alpha_{n,i_{n}} }}  $  belonging to $\mathcal{F}_{n} (\alpha)$ are embedded in those of $\mathcal{F}_{n-1} (\alpha)$, and (remembering definition \eqref{defini}), these intervals are separated by a distance at least equal to 
$$2^{\left [J_{n}(1- \frac{  \gamma_{{n},i_{n}} }{\alpha_{{n},i_{n}}} )\right]} 2^{-J_{n}}  \sim 2^{ - J_{n}  \frac{  \gamma_{{n},i_{n}} } {\alpha_{{n},i_{n}}} } .$$

By a classical argument (see \cite{F2}, examples 4.6 and 4.7 for instance)  allowing to compute the Hausdorff dimension of this type of Cantor sets,  we have
\begin{eqnarray*}
 \dim_{\mathcal{H}} \, \mathcal{F}(\alpha) &  \geq  &  \liminf_{n\geq +\infty}    \frac {  -  \log \, ( \, \prod_{m=1}^n \ \# \mathcal{T}_{m,i_{m}})} {    \log \,   2^{-J_{n} /{\alpha_{n,i_{n}} }} }
\\
 & =  & \liminf_{n\geq +\infty}    \frac {   -\log  \ \# \mathcal{T}_{n,i_{n}} } {    \log \,   2^{-J_{n} /{\alpha_{n,i_{n}} }} }\\
 & = & \lim_{n\to+\infty}  {\gamma_{n,i_{n}}   } (1-\ep_{n,i_{n}})  = d(1+\eta\alpha).
 \end{eqnarray*}

Suppose $\aaa_{0}<\aaa\leq\bbb_{0}$ and $0<r<2^{-J_{5}-1}$ and choose
$N$ satisfying \eqref{defJN}. Then $x\in \cae_{N}(\aaa)=\bigcup_{k\in \mathcal{T}_{N,i _{N}}} I_{J_{N},k} (\alpha_{N,i_{N}})$
and we have $\cae_{N}(\aaa)\cap E''_{N,i_{N},r}=\ess$.
This implies $x\not\in E_{N,i_{N},r}(Z)$, that is,
\begin{equation*}
\ooo_{B(x,r)}(Z)<(2r)^{\aaa_{N,i_{N}}}.
\end{equation*}
If $\aaa=\aaa_{0}$ by \eqref{*19201*} we have 
$\ooo_{B(x,r)}(Z)<(2r)^{\aaa_{0}}$.

On the other hand,
$\cae_{N}(\aaa)\sse E'_{N,i_{N}+1}$ and there exists
$r'\leq 2^{-J_{N}/\aaa_{N,i_{N}+1}}\leq 2^{-J_{N}}$
 such that 
 \begin{equation*}
 \ooo_{B(x,r')}(Z)\geq (2r')^{\aaa_{N,i_{N}+1}}.
 \end{equation*}
 Since $\aaa_{N,i_{N}+1}-\aaa_{N,i_{N}}\to 0$ as $N\to\oo$
 and $\aaa_{N,i_{N}}\leq \aaa <\aaa_{N,i_{N}+1}$ we have
 $h_{Z}(x)=\aaa.$

%%%%%%%%%%%%%%%%%%%%%%%%%%%%%%%%%%
%%%%%%%%%%%%%%%%%%%%%%%%%%%%%%%%%%
%%%%%%%%%%%%%%%%%%%%%%%%%%%%%%%%%%
%%%%%%%%%%%%%%%%%%%%%%%%%%%%%%%%%%
%%%%%%%%%%%%%%%%%%%%%%%%%%%%%%%%%%
%%%%%%%%%%%%%%%%%%%%%%%%%%%%%%%%%%
%%%%%%%%%%%%%%%%%%%%%%%%%%%%%%%%%%
%%%%%%%%%%%%%%%%%%%%%%%%%%%%%%%%%%
%%%%%%%%%%%%%%%%%%%%%%%%%%%%%%%%%%
%%%%%%%%%%%%%%%%%%%%%%%%%%%%%%%%%%
%%%%%%%%%%%%%%%%%%%%%%%%%%%%%%%%%% 

\section{Multifractal  spectrum  prescription of a non-HM measure}

\label{sec_cons11}

Let $f\in  \mathcal{F}$. We build a probability measure $\mu$ whose multifractal spectrum $d_\mu(h)$ is exactly $f(h)$ for the exponents $h<1$. To get  part (iii) of Theorem \ref{th111} (the spectrum for $h=1$), it is enough to consider the measure $(\mu + \lll)/2$, where $\lll$ is the Lebesgue measure on $\zu$.

\mk

 We call $(\tilde f_n)$ the sequence of functions associated with $f\in \mathcal{F}$, which are  constant  
 over a closed interval $I_n \subset \zu$ 
 and satisfy $\widetilde f_n(x) \leq f(x)$. We set 
\begin{equation}\label{*aodef}
\alpha_0 = \inf_{n\geq 1} \min I_n>0.
\end{equation}

Recall that $\mathcal{F}^*$, introduced in Definition \ref{defFetoile}, was the set of functions suitable for Theorem \ref{proplinearspectrum}.
For each $\widetilde f_n$, there exists a countable sequence  of affine functions $(f_{n,p})_{p\geq 1}$,
$f_{n,p}\in \caf^{*}$
 such that
$\text{Support}(f_{n,p})=I_{n}$ for all $p$
and 
 $\widetilde f_n = \sup_{p} f_{n,p}$. 

Hence, if we consider the  countable family of  functions $(f_{n,p})_{n\geq 1,\  p\geq 1}$, we  still have
 $f=\sup_{n\geq 1,\ p\geq 1} f_{n,p}.$ 
 By adjusting our notation
 we call this new countable family    $ (   f_p)_{p\geq 1}$, and we have $f = \sup_{p\geq 1}    f_p$.

\begin{remark}
This procedure will be used in the next section also.
\end{remark}

For every integer $p\geq 1$, by Theorem \ref{proplinearspectrum}, one can find a surjective monotone function  $Z_p:\zu\to\zu $ whose multifractal spectrum is exactly $f_p$. Let us call $\mu_p$ the measure defined as the integral of $Z_n$: $\mu_p([0,x])= Z_p(x)$. We obtain the measure $\mu$  as follows: for every $p\geq 1$, the restriction of $\mu$ to the interval $[2^{-p}, 2^{-p+1}]$ coincides with  $2^{-p} (\mu_p\circ \ell_p)$, where $\ell_p$ is the unique affine bijective increasing map from $\zu $ to $[2^{-p}, 2^{-p+1}]$.

It is a trivial matter to see that  
$$d_\mu = \sup_{p\geq 1} d_{\mu_p},$$
 since affine mappings do not modify the multifractal spectrum and since the supports of the measures $\mu_p\circ \ell_p$ are disjoint.  Problems may occur only at the concatenation points between the supports of the measure, i.e. at the rationals $2^{-p}$, $p\geq 1$, and at $0$. In reality, there is no problem at the rationals  $2^{-p}$, because of item (i) of Theorem \ref{proplinearspectrum} (and the fact that 0 and 1 are isolated points  and  do not belong to the support of the measures  $\mu_p$), which ensures that these points satisfy $h_\mu(2^{-p})=+\infty$. Hence, only 0 may be a problem, but it is easy  to modify the measure $\mu$ (for instance by performing a time subordination which is singular at 0) to ensure that $h_\mu(0)=+\infty$. This yields Theorem \ref{th111}.

%%%%%%%%%%%%%%%%%%%%%%%%%%%%%%%%%%
%%%%%%%%%%%%%%%%%%%%%%%%%%%%%%%%%%
%%%%%%%%%%%%%%%%%%%%%%%%%%%%%%%%%%
%%%%%%%%%%%%%%%%%%%%%%%%%%%%%%%%%%
%%%%%%%%%%%%%%%%%%%%%%%%%%%%%%%%%%
%%%%%%%%%%%%%%%%%%%%%%%%%%%%%%%%%%
%%%%%%%%%%%%%%%%%%%%%%%%%%%%%%%%%%
%%%%%%%%%%%%%%%%%%%%%%%%%%%%%%%%%%
%%%%%%%%%%%%%%%%%%%%%%%%%%%%%%%%%%
%%%%%%%%%%%%%%%%%%%%%%%%%%%%%%%%%%
%%%%%%%%%%%%%%%%%%%%%%%%%%%%%%%%%% 
\section{Multifractal  spectrum  prescription of a  HM measure }
\label{sec_cons3}

Let $f$ be a function belonging to the set of functions $\caf$. We apply the same procedure as in the previous Section \ref{sec_cons11} to get a  countable family of  functions     $ (   f_p)_{p\geq 1}$ all belonging to $\mathcal{F}^*$ and satisfying  $f = \sup_{p\geq 1}    f_p$.

\mk

For each function $   f_p$,  using Theorem \ref{proplinearspectrum}, there exists a function $   Z_p$ whose spectrum 
on $[0,+\oo)$
is exactly $   f_p$. The construction of Theorem \ref{proplinearspectrum} guarantees that the function $   Z_p$ has a particular form: the points $x$ such that $h_{   Z_p}(x) <+\infty$ are located on a Cantor set $C_p$, and the largest interval  in the complementary set of $C_p$ in $\zu$ has length less than 1/10. The functions $Z_n$ are extended as continuous functions $\R\to \R$, by setting $Z_n(x)=0$ if $x\leq 0$ and $Z_n(x)=1$ if $x\geq 1$.
 
 \mk

The idea behind our construction is to "insert" in each open interval complementary to the Cantor set $C_p$ a copy of another function $   Z_{p'}$, so that the new function will have a  multifractal  spectrum equal to the supremum of the spectra of $   f_p$ and $   f_{p'}$, since the two functions have disjoint supports.
 We will repeat this a countable number of times, with a strong redundancy, and the function $Z$ obtained 
 as the uniform limit of  continuous functions $(Y_{n})$ will have the desired homogeneous multifractal spectrum.

\mk

We will first construct the HM  function $Z$ as the uniform limit of a sequence of functions $(Y_n)_{n\geq 1}$,  using  a suitable
 subsequence of functions $(g_n)_{n\geq 1}$ which will be selected from the 
 set of functions $(   f_p)_{p\geq 1}$. For the moment, we do not explain how this choice is made, we will do it at the end of this section.  By abuse of notation, we still denote by $Z_n$ the function built in Theorem \ref{proplinearspectrum} whose spectrum equals $g_n$.
Afterwards, we will explain how we choose each function $g_n$ among the functions $(   f_p)_{p\geq 1}$ in order to impose a homogeneous  multifractal  spectrum for $Z$.

\mk

We denote by $Z_{0}$ the function obtained from Proposition \ref{homone}.

Set $Y_1=Z_{0}+Z_1$. Then for all $x\in [0,1]$ we have $h_{Y_{1}}(x)\leq 1$
and the set of singularities  $\widetilde C_1=\{x\in \zu: h_{Y_1}(x) <1\}$
is located on a Cantor set.

%%%%%%%%%%%%%%%%%%%%%%%%%%%%%%%%%% 
\begin{center}
\begin{figure}
  \includegraphics[width=7.cm,height=6cm]{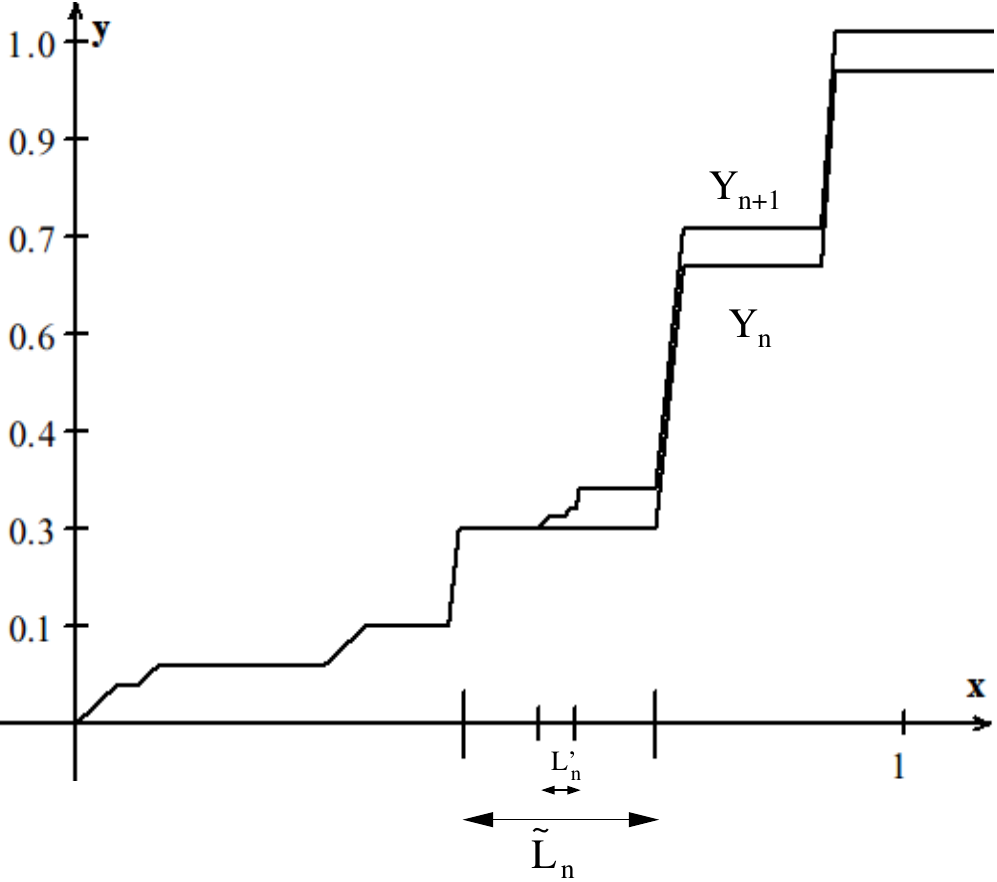}
\caption{Definition of $Y_{n}$ and $Y_{n+1}$} \label{figpsp5c}
\end{figure}
\end{center}
%%%%%%%%%%%%%%%%%%%%%%%%%%%%%%%%%% 

For every $n\geq 2$, we assume that $Y_n$ has been  built, and that the set $\widetilde C_n$ of singularities of $Y_n$, i.e. $\widetilde C_n=\{x\in \zu: h_{Y_n}(x) <1\}$ has the structure of a Cantor set : there exists a sequence of sets $(C_{n,p})_{p \geq 1}$ satisfying the following:
\begin{itemize}
\item
the $C_{n,p}$ are nested, i.e. for every $p\geq 1$, $C_{n,p+1} \subset C_{n,p}$,
\item
$C_{n,p}$ is a finite union of pairwise disjoint  closed intervals,
\item
the maximal length of the intervals in $C_{n,p}$ is strictly decreasing with $p$ and  tends to zero when $p$ tends to infinity,
\item
$\widetilde C_n = \bigcap_{p\geq 1} C_{n,p}$. 
\end{itemize}
We also assume that $h_{Y_{n}}(x)\geq \aaa_{0}$ for all $x\in [0,1].$

Then we construct $Y_{n+1}$ as follows: let $\widetilde {L}_n$ be one of the longest open intervals contiguous to $\widetilde C_n$ in $\zu$, and let $L_n'$ be concentric with $\widetilde {L}_n$ but of length $2^{-n^2}$ times  that of $\widetilde {L}_n$, i.e. $|L'_n|=2^{-n^2}|\tilde L_n|$. 

We set (see Figure \ref{figpsp5c})
\begin{equation}
\label{defzn2}
\forall \ x\in\zu, \ \ \ Y_{n+1} (x) = Y_n(x) +  2^{- n^2/ |\widetilde {L}_n |}  \cdot Z_{n+1} (S_n(x)) , 
\end{equation}
where $S_n$ on $L_n'$ is the unique increasing affine function mapping $L_n'$ to $\zu$, otherwise $S_{n}$ is continuous and constant on the components
of $\zu \sm L_n'$.

 By \eqref{minorZ} of Theorem \ref{proplinearspectrum} we have 
$h_{Y_{n+1}}(x)\geq \aaa_{0}$ for all $x\in [0,1]$.

Then $Y_{n+1}$ obviously satisfies the same properties as $Y_n$: its set of singularities $\widetilde C_{n+1}=\{x\in \zu: h_{Y_{n+1}}(x) <1\}$ has the structure of a Cantor set, since it is the union of two Cantor sets ($\widetilde C_n$ and the image of $C_{n+1}$ by $S_n^{-1}$) which ``do not cross", i.e. between any two  points of  the image of $C_{n+1}$ by $S_n^{-1}$, there is no point of $\widetilde C_n$. 
Moreover, 
in case $\widetilde C_n$ had only one contiguous interval of maximal size, $\widetilde L_n$
then the  size of one of the largest open interval in the complement of $\widetilde C_{n+1}$ is less than that of the largest open intervals in the complement of $\widetilde C_{n}$, since the interval $\widetilde {L}_n$ has been cut in many parts (at least 2).
Otherwise  the (finite) number of contiguous intervals to $\widetilde C_{n+1}$
of size $\widetilde L_n$ is one less than that for $\widetilde C_n$.

\mk

We iterate this construction.

%%%%%%%%%%%%%%%%%%%%%%%%%%%%%%%%%%
\begin{proposition}
The sequence of functions $(Y_n)_{n\geq 1}$ converges uniformly to a continuous function $Z:\zu\to\R$.
\end{proposition}
%%%%%%%%%%%%%%%%%%%%%%%%%%%%%%%%%%
%%%%%%%%%%%%%%%%%%%%%%%%%%%%%%%%%%
\begin{proof}
Using the definition of $Y_n$  in \eqref{defzn2},  this follows from the fact that 
$0<|\widetilde{L}_n|<1$ is non-increasing, and that 
the series $\sum _{n\geq 1} 2^{- n^2/ |\widetilde {L}_n |} $ converges.  \end{proof}
%%%%%%%%%%%%%%%%%%%%%%%%%%%%%%%%%%

In addition, the sequence of lengths $(|\widetilde {L}_n|)_{n\geq 1}$ is not only non-increasing, but it also tends to zero. Indeed, at a fixed step $n$, the set of intervals in the complementary set of $\widetilde {C}_n$ whose size is between  $|\widetilde L_n|$  and $|\widetilde L_n|/2$ is finite. Since any further step divides by  more than two the size of (at least one of) the maximal interval(s), after a finite number of steps, the size  of the maximal interval(s) in the complement of  $\widetilde {C}_{n+n'}$ will be less than $|\widetilde L_n|/2$. Hence, $(|\widetilde {L}_n |)_{n\geq 1}$ converges to zero when $n$ tends to infinity.

\mk

Since we are adding monotone functions the oscillation of $Y_n$ on a given interval $I$ can only increase when $n$ increases. One  consequence is that for each $x \in   \widetilde C_N$, the  \ho exponent of $Z$ at $x$ is not larger than the \ho exponent of $Y_N$ at $x$.

Moreover, using \eqref{minorZ}, all the \ho exponents of $Z$ are larger than $\alpha_0$.

Hence the multifractal spectrum of $Z$ has a support included in  $[\alpha_0,+\infty]$.

\mk

We now prove the key proposition to obtain Theorem \ref{th1}: it asserts that the set  of those points where the  \ho exponent can be altered during the iterative construction of the functions $Y_{n}$ has Hausdorff dimension $0$.

%%%%%%%%%%%%%%%%%%%%%%%%%%%%%%%%%%
\begin{proposition} For every $N\geq 2$, the  Hausdorff dimension of the set
$$
\widetilde{\mathcal{F}}_{N}=
 \{x\in [0,1] : h_Z(x) < h_{Y_N}(x) \}$$
is zero. 
\end{proposition}
%%%%%%%%%%%%%%%%%%%%%%%%%%%%%%%%%%
%%%%%%%%%%%%%%%%%%%%%%%%%%%%%%%%%%
\begin{proof}

Let us choose an $x\in [0,1]$ with $h_{Z}(x)<h_{Y_{N}}(x)\leq 1$. This means that $h_{Y_N}(x)$ has a  value, greater than $\alpha_0$ and not exceeding $1$. Let $n>N$. The maximal value of the oscillation of the contribution of $Y_n$ is  $2^{-n^2/|\widetilde L_n|}$. Moreover, all the further contributions of the $Y_{n'}$, for $n'\geq n$, are all of magnitudes  less than $2^{-(n')^2/|\widetilde L_{n'}|}$ (which itself is  less than $2^{-(n')^2/|\widetilde L_{n}|}$),  so the sum of all the maximal oscillations 
is less than  $2\cdot 2^{-n^2/|\widetilde L_n|}$.

One knows that for every $r$ small enough, 
\begin{equation}\label{*oohyn}
\omega_{B(x,r)}(Y_{N}) \leq r^{h_{Y_N}(x)-\ep}.
\end{equation}
 We assumed that  $ h_Z(x) < h_{Y_N}(x)$, and let $\ep >0$ be such that $ h_Z(x)  +3\ep < h_{Y_N}(x) $. Necessarily, for some small values for $r$,  one must have 
$$\omega_{B(x,r)}(Z) \geq r^{h_{Z }(x)+\ep}> r^{h_{Y_{N}}(x)-2\eee}>
2 r^{h_{Y_{N}}(x)-\eee}\text{ and }r^{h_{Y_{N} }(x)+\ep}/2\geq r^{h_{Y_{N}}(x)+2\eee}.$$ This 
and \eqref{*oohyn}
imply that for some small values of $r>0$
\begin{equation}
\label{eq4}
\om_{B(x,r)} ( Z-Y_N) \geq r^{h_{Z }(x)+\ep}/2 \geq r^{h_{Z }(x)+2\eee}>r.
\end{equation}

In order to modify the oscillation of $Y_{N}$ on $B(x,r)$, the ball
  $B(x,r)$ should intersect at least one of the intervals 
$L_n'$ for an $n\geq N+1$. 
Let $n\geq N+1$ be the minimal integer such that 
$B(x,r ) \cap L_n'\neq \emptyset$.   
The maximal  possible value of the oscillation of the function
$2^{- n^2/ |\widetilde {L}_{n} |}  \cdot  Z_{n+1} (S_{n}(x))$ (which is 
added at step $n+1$ to construct $Y_{n+1}$ from $Y_{n}$)  equals 
 $2^{- n^2/ |\widetilde {L}_{n} |} $.
This oscillation is obtained  on the interval $L_{n}'$  of length $2^{- n^2}|\widetilde{L}_{n}|$. The difference between $\om_{B(x,r)}(Y_{n})$ and  $\om_{B(x,r)}(Y_{n+1})$ 
is at most
$2^{- n^2/ |\widetilde {L}_{n } |} $. Using the remark above,  summing all the further oscillations (for $n'\geq n+1$) does not change much the size of $\om_{B(x,r)}(Y_{n})$.

Assume that $r> |L_n'| = 2^{- n^2}|\widetilde{L}_{n}|$. 
Also recall that $|\widetilde{L}_{n}|\leq 1/10.$
The difference between $\om_{B(x,r)}  (Y_{n})  $ and  $\om_{ B(x,r)}(Z )$ 
is at most
$$2\cdot 2^{- n^2/ |\widetilde {L}_{n } |}       <\!\!\!<
2^{-n^{2}+\log_{2}|\widetilde{L}_{n}|}=
 2^{- n^2}|\widetilde{L}_{n } |  = |L_n'| \leq r  . $$

Therefore, we have  $|\om_{ B(x,r)}(Z ) - \om_{B(x,r)}  (Y_{n}) | \leq r$, $|\om_{ B(x,r)}(Y_n ) - \om_{B(x,r)}  (Y_{N}) |  = 0$.
 This contradicts  \eqref{eq4}.
 
Hence, we need to consider the case $r\leq |L_n'|$ and $B(x,r)\cap L_n'\not=\ess$.
This
 implies that $x$ must belong to the interval concentric with $L_n'$, but of
 length $3 |L_n'|$. Let us call these intervals $L_n''$.

In order to change at $x$ the oscillation of $Z$  compared to that of $Y_{N}$,  the point
 $x$ must belong to  an infinite number of intervals   $L_n''$. Let us now find an upper-bound for the dimension of the set $\mathcal{S}$ of such points.
 For any $s>0$,  if $\eta_n = |L_n''|$, the union $\bigcup_{n'\geq n} L_{n'}''$  forms 
 an $\eta_n$-covering of $\mathcal{S}$, thus  we have
$$\mathcal{H}^s_{\eta_n} (\mathcal{S}) \leq \sum_{n'\geq n}| L_{n'}''|^s \leq  3^s \sum_{{n'}\geq n }| 2^{- {n'}^2/ |\widetilde {L}_{{n'} }|} |^s \leq  3^s \sum_{{n'}\geq n}| 2^{- {n'}^2/ |\widetilde {L}_{N } |} |^s  . $$
This last sum obviously converges for any value of $s>0$. Hence $\mathcal{H}^s  (\mathcal{S}) =0$ and $\dim(\mathcal{S}) \leq 0$.
Since $\widetilde{\mathcal{F}}_{N}\sse \mathcal{S}$ we proved the proposition.
\end{proof}
%%%%%%%%%%%%%%%%%%%%%%%%%%%%%%%%%%

By taking a countable union of sets of dimension zero, we obtain:

%%%%%%%%%%%%%%%%%%%%%%%%%%%%%%%%%%
\begin{corollary} Let $\mathcal{F}= \cup_{N=2}^{\oo}\widetilde{\mathcal{F}}_{N}$. Then $\mathcal{F}$ is
the subset of $[0,1]$ for which the exponent of $Z$ at $x$ is modified due to our ``scheme of iteration" and  $\dim \mathcal{F} = 0$.\end{corollary}
%%%%%%%%%%%%%%%%%%%%%%%%%%%%%%%%%%

\mk

We finish by explaining the choice of the functions $(f_n)_{n\geq 1}$ in the construction of the function $Z$. This sequence is obtained recursively:
\begin{itemize}
\item
{\bf Step $1$:} We use $g_n=   f_1$ until each dyadic interval $I_{1,k}$, for $k=0,1$ contains a copy of $Z_1$.

\item
{\bf Step $2$:} We use $g_n = f_1$ until each dyadic interval $I_{2,k}$, for $k=0,\cdots, 2^2-1$ contains a copy of $Z_1$. Then we use $f_n = g_2$ until each dyadic interval $I_{2,k}$, for $k=0,\cdots, 2^2-1$ contains a copy of $Z_2$.

\item
...

\item
{\bf Step $p$:} We use $g_n = f_1$ until each dyadic interval $I_{p,k}$, for $k=0,\cdots, 2^p-1$ contains a copy of $Z_1$. Then we use $g_n = f_2$ until each dyadic interval $I_{p,k}$, for $k=0,\cdots, 2^p-1$ contains a copy of $Z_2$ ....  Finally we use $g_n = f_p$ until each dyadic interval $I_{p,k}$, for $k=0,\cdots, 2^p-1$ contains a copy of $Z_p$. 
\item ...
\end{itemize} 
At the end of the construction, we obviously have the following property: Any non-trivial interval $I\subset \zu$ contains a copy  of any function $f_p$. 
If $h\in \text{Support}(f)$ and $f(h)>0$ then from $\dim \mathcal{F} =0$
it follows that $d_{Z}(h)=f(h).$ If $f(h)=0$ and $h\geq \aaa_{0}$
then Theorem \ref{thdarbouxb} can be used to verify that $d_{Z}(h)=f(h)=0.$ Indeed, $h_{Z}(x)\leq h_{Y_{N}}(x)$ for any $x\in [0,1]$
and $N\in \N$. By the choice of $\aaa_{0}$ in \eqref{*aodef}
and by construction $\{ x:h_{Z}(x)\leq \aaa_{0}+\eee \}$
is dense in $[0,1]$
 for any $\eee>0$. Moreover, $h_{Z}(x)\geq \aaa_{0}$ for any
 $x\in [0,1].$
 The existence and the value of the rest of the spectrum (i.e. the exponents $h$ for which $d_Z(h) = 0$) are obtained  by combining  the results of the previous section.

%%%%%%%%%%%%%%%%%%%%%%%%%%%%%%%%%%
%%%%%%%%%%%%%%%%%%%%%%%%%%%%%%%%%%
%%%%%%%%%%%%%%%%%%%%%%%%%%%%%%%%%%
%%%%%%%%%%%%%%%%%%%%%%%%%%%%%%%%%%
%%%%%%%%%%%%%%%%%%%%%%%%%%%%%%%%%%
%%%%%%%%%%%%%%%%%%%%%%%%%%%%%%%%%%
%%%%%%%%%%%%%%%%%%%%%%%%%%%%%%%%%% 
\section{Multifractal  spectrum  prescription of a HM (non-monotone) function}
\label{sec_cons4}

Theorem \ref{th3} is simply the consequence of the following  Theorem proved in \cite{BS1}. Let $(\psi_{j,k})_{j\in \Z,k\in \Z}$ be any orthogonal  wavelet basis of $L^2(\R)$ (see for instance \cite{MEYER} for the existence and the construction of wavelet bases).

%%%%%%%%%%%%%%%%%%%%%%%%%%%%%%%%%% 
\begin{theorem}
\label{thbs}
Let $\mu$ be a measure on $\zu$,    $0<\alpha <\beta$. Consider the wavelet series
$$F_\mu(x) = \sum_{j\geq 1} \sum_{k=0}^{2^j -1}  \  d_{j,k} \  \psi_{j,k}(x),$$
where the wavelet coefficients of $F_\mu$ are defined by
$d_{j,k} = 2^{- j \alpha} \mu\left(   I_{j,k} \right)^{\beta -\alpha}.$

Then for every $x\in \zu$, 
 $h_{F_\mu}(x)  =  \alpha+ (\beta -\alpha) \cdot h_\mu(x).$
 
This implies that for every exponent $h$ such that  $\frac{h-\alpha}{\beta-\alpha} $ belongs to $  \mbox{Support\,($d_\mu$)}$, one has
$ d_{F_\mu}(h) = d_\mu \left( \frac{h-\alpha}{\beta-\alpha} \right).$
\end{theorem}
%%%%%%%%%%%%%%%%%%%%%%%%%%%%%%%%%% 

Theorem \ref{thbs} combined with Theorem \ref{th11} yields Theorem \ref{th3}.

%%%%%%%%%%%%%%%%%%%%%%%%%%%%%%%%%%
%%%%%%%%%%%%%%%%%%%%%%%%%%%%%%%%%%
%%%%%%%%%%%%%%%%%%%%%%%%%%%%%%%%%%
%%%%%%%%%%%%%%%%%%%%%%%%%%%%%%%%%%
\section{A HM measure with a spectrum gap when $h>1$}
\label{sec_noncontinuous}

%%%%%%%%%%%%%%%%%%%%%%%%%%%%%%%%%%
%%%%%%%%%%%%%%%%%%%%%%%%%%%%%%%%%%
%%%%%%%%%%%%%%%%%%%%%%%%%%%%%%%%%%
%%%%%%%%%%%%%%%%%%%%%%%%%%%%%%%%%%
%%%%%%%%%%%%%%%%%%%%%%%%%%%%%%%%%%
%%%%%%%%%%%%%%%%%%%%%%%%%%%%%%%%%%
%%%%%%%%%%%%%%%%%%%%%%%%%%%%%%%%%% 

We prove Proposition \ref{*spectrgap}. Our construction is very technical, but we do not know any other "easy" construction of such a measure. Actually we will rather build a continuous increasing HM function.

 \begin{definition} 
 Given a non-degenerate closed interval $J=[a,b]$ we call $\fff:J\to\R$ a Cantor-type
interpolating function if $\fff$ is continuous,   non-decreasing and
there exists a closed set $E_{\fff}$ of zero Lebesgue measure  
such that $\fff$ is constant on the intervals contiguous to
$E_{\fff}$. Moreover, it has a multifractal spectrum
as bad as possible, that is $d_{\fff}(h)=h$, for all $h\in [0,1]$.
 \end{definition}
 
 \begin{remark}
 It is not too difficult to provide direct constructions of such functions,
 but using Theorem \ref{proplinearspectrum} we can easily verify their existence.
  Indeed, consider sequences $\aaa_{0,n}$, $\bbb_{0,n},
 $ $d_{n}$ and $\eta_{n}$ such that all possible rational values of parameters
 satisfying the assumptions of Theorem \ref{proplinearspectrum} appear among them.
Denote by $Z_{n}$ the sequence of monotone continuous functions  which we obtain
from Theorem \ref{proplinearspectrum} by using these parameter values. Set $\fff(a)=0$ and $$\fff(x)=\frac{1}{2^{n-1}}Z_{n}\Big (\Big (\frac{b-a}{n}-\frac{b-a}{n+1}\Big )^{-1}\Big (x-\frac{b-a}{n+1}\Big )\Big )$$ $$\text{ for }
x\in \Big (\frac{b-a}{n+1},\frac{b-a}{n}\Big ], \  n=1,2,... .$$
\end{remark}

\begin{definition} 
A finite  set of real numbers  $S =\{ s_{1},...,s_{n} \}\sse (0,1)$ is said to be $\ddd $-discrete (with $\ddd >0$) if
\begin{align}\label{*V2*2}
&\text{the distance between any two intervals $[s_{i}-\ddd ,s_{i}+\ddd ]$, $i=1,...,n$,}\\ \nonumber
&
\text{is larger than $2\ddd $, and the distance of these intervals from $0$, or from $1$}\\ \nonumber
&\text{is also larger than $2\ddd$.}
\end{align}

\end{definition}

Now we are turning to the proof of Proposition \ref{*spectrgap}.
  We are going to construct the HM measure $\mmm$ by defining its
Borel integral $F(x)=\int_{0}^{x}d\mmm$.
The  function
$F$ will be an infinite  sum of monotone increasing continuous functions $F_{n}$.

\subsection{First part of the definition of the functions $F_{n}$ by induction}

We introduce several sequences of  numbers, sets and functions, which will be the basis of our induction.

First,  we fix a sequence of intervals $(I_{n}:=[a_{n},b_{n}])_{n\geq 1}$ satisfying $I_n \sse I_1:=(0,1)$, 
$b_{n}-a_{n}\searrow 0$ and $\{ a_{n}:n=1,... \}$
is dense in $[0,1]$. The sequence of intervals $(I_n)$ is thus also dense.

\medskip

We start with  $\ddd_{0}=1$, $S_{0}=\ess$, $H_{0}=\ess$,
$I_{1}'=[0,1]$, $T_{0}=\ess$, $\widetilde I_{1}=I_{1}$.

We assumme that for some integer $n\geq0$, we have built the following:
\begin{enumerate}

\smallskip \item  $n+1$ real numbers $\ddd_{0}>...>\ddd_{n} >0$, satisfying for $1\leq p\leq n$
\begin{equation}\label{*V8*2}
(  {2^{-p}}+
p )    \ddd_{p}^{2}  <\frac{\ddd_{p-1}^{2}}{2}.
\end{equation}
(We remark that for the case $n=0$ here and later there
is no $p$ satisfying $1\leq p\leq n$, which means that in the $n=0$ case
these assumptions are not needed.)

\smallskip\item
for $1\leq p\leq n$, an increasing sequence of   finite sets  of points $S_{p}=\{ s_{1},...,s_{p} \}\sse (0,1)$  that are $\ddd_{p}$-discrete and such that 
$s_{p}\in I_{p}$. We put
\begin{equation}
\label{defHi}
H_{p}=\bigcup_{i=1}^{p}  \ [ s_i-\ddd_{p},
s_{i}+\ddd_{p}] .
\end{equation}

\smallskip\item
 closed intervals 
$(\widetilde I_{p})_{p = 1,...,n+1}$, such that $\widetilde I_{p} \sse I_{p}$ and  for $1\leq p \leq n$, 
\begin{equation}\label{*V2*1}
\text{$\widetilde I_{p+1} \sm H_{p}$ and $I_{p+2}\sm H_{p}$ contain closed
intervals of length larger than $2\ddd_{p}$}. 
\end{equation}

\smallskip\item     monotone non-decreasing continuous functions $F_{p}:\zu\to \zu$, $1\leq p\leq n$ satisfying:

\begin{enumerate}

\smallskip\item  $F_{p}$ is constant on the intervals $[s_{i}-\ddd_{p},s_{i}+\ddd_{p}/2]$, $i=1,...,p$,  

 \smallskip\item  on the intervals $[s_{i}+\ddd_{p}/2,s_{i}+\ddd_{p}]$, $i=1,...,p$,  
  the function $F_{p}$ coincides with a Cantor-type interpolating function
 whose increment on this interval is given by
\begin{equation*}
F_{p}(s_{i}+\ddd_{p})-F_{p}(s_{i}+\ddd_{p}/2)=\ddd_{p}^{2}\text{ for $i=1,...,p$}.
\end{equation*}
We call    $T_{p,i}$ the nowhere dense closed set associated with the restriction of $F_p$ to  $[s_{i}+\ddd_{p}/2,s_{i}+\ddd_{p}]$, and   $T_p= \bigcup_{i=1}^p T_{p,i}$. By construction,  $F_{p}$ is   constant on the intervals
contiguous to    $T_p$.

  \smallskip\item for all $x\not \in H_{p}$    there exists an interval $I_{x,p}=
[a_{x,p},b_{x,p}]$  such that 
\begin{equation}\label{*V3*2}
x\in I_{x,p},\ b_{x,p}-a_{x,p}<\ddd_{p}  \ \text{ and } \ 
F_{p}(b_{x,p})-F_{n}(a_{x,p})\geq(b_{x,p}-a_{x,p})^{1+\frac{1}{p}}.
\end{equation}
 
 \smallskip\item  for $1\leq p <i \leq n $, $F_{p}$ is constant on the  intervals of $H_{i}$  (see \eqref{defHi}). In other words, $H_n$ "avoids" the Cantor sets $T_p$, $p\leq n-1$ of zero Lebesgue measure where the functions $F_p$ increase.

\smallskip\item
finally, for $1\leq p \leq n$, 
\begin{equation}\label{*V8*3bis}
F_{p}(1)-F_{p}(0)\leq p\ddd_{p}^{2}+\frac{\ddd_{p}^{2}}{2^{p}}.
\end{equation}

\end{enumerate}

\end{enumerate}

Observe that  part {(iii)}  implies that  $S=\cup_{n}S_{n}$ is dense in $[0,1]$.

\subsection{The next step of the induction}

Suppose $n\geq 0$. We need to define $s_{n+1}$, $\ddd_{n+1}$, $\widetilde I_{n+2}$, $T_{n+1}$
and $F_{n+1}$. 
We assume that 
\begin{equation*}
\frac{\ddd_{n+1}^{2}}{2^{n+1}}+
(n+1)\ddd_{n+1}^{2}<\frac{\ddd_{n}^{2}}{2}.
\end{equation*}

Using \eqref{*V2*1} we  select  a closed subinterval  $\widetilde I_{n+2}\subset I_{n+2}\sm H_{n}$ 
of length $2\ddd_{n}$. Then, 
\begin{equation}\label{*sn1val*}
\text{in the interior of $\widetilde  I_{n+1}\sm H_{n}$, we select a point
$s_{n+1}\not \in T_{n}$.} 
\end{equation} 
Hence, by choosing a sufficiently small $0 < \ddd_{n+1}<\ddd_{n}/2$, 
we can ensure that for all $1\leq p \leq n$ the functions $F_{p}$
are constant on $[s_{n+1}-\ddd_{n+1},s_{n+1}+\ddd_{n+1}]
\sse [0,1]\sm T_{n}$, and thus  
on $H_{n+1}$.

By \eqref{*sn1val*}, 
we   suppose that $\ddd_{n+1}$ is also so small that
\begin{equation}\label{*V5*2}
[s_{n+1}-\ddd_{n+1},s_{n+1}+\ddd_{n+1}]\sse \intt(\widetilde I_{n+1}\sm H_{n}).
\end{equation}

\medskip

Now we build $F_{n+1}$. Set $F_{n+1}(0)=0$. 

Next we  define $F_{n+1}$ on  $\zu \setminus H_{n+1}$, which is constituted by  finitely many
intervals contiguous
to $H_{n+1}\cup \{ 0 \}\cup \{ 1 \}$.  Observe that by choosing a sufficiently small $\ddd_{n+1}$,  we can 
ensure that for all intervals 
$[\aaa,\bbb]$ contiguous to $H_{n+1}\cup \{ 0 \}\cup \{ 1 \}$, $\bbb-\aaa>2\ddd_{n+1}$. In addition, we can also suppose
that both $\widetilde I_{n+2}\sm H_{n+1}$ and $I_{n+3}\sm H_{n+1}$
contain a closed interval of length $2\ddd_{n+1}$.

Now fix an interval  $J=[\aaa,\bbb]$ contiguous to $H_{n+1}\cup \{ 0 \}\cup \{ 1 \}$. We pick an integer  $\kkk\in\N$  such that 
\begin{equation}\label{*V5*1}
\frac{\bbb-\aaa}{\kkk}\leq \frac{\ddd_{n+1}^{2(n+1)}}
{2^{(n+1)^{2}}}
\end{equation} 
and subdivide $J$ into subintervals 
\begin{equation}
\label{defjl}
J_{l}=\Big [\aaa+(l-1)\frac{\bbb-\aaa}{\kkk},\aaa+l\frac{\bbb-\aaa}{\kkk}\Big ]=[\aaa_{l},\bbb_{l}],\ l=1,...,\kkk.
\end{equation}
We define $F_{n+1}$ so that:
\begin{itemize}
\item the increments on $J_l$ are 
\begin{equation}\label{*V6*1}
F_{n+1}(\bbb_{l})-F_{n+1}(\aaa_{l})=
(\bbb_{l}-\aaa_{l})\frac{\ddd_{n+1}^{2}}{2^{n+1}}, \text{ for all }l=1,...,\kkk.
\end{equation}

\item On the interior of $J_{l}$,    $F_{n+1}(x)-F_{n+1}(\aaa_{l})$  is  a Cantor type 
interpolating function $\fff_{l}$ and \eqref{*V6*1} is also satisfied.
\end{itemize}

From \eqref{*V5*1} it follows that
$$(\bbb_{l}-\aaa_{l})^{\frac{1}{n+1}}=\Big (\frac{\bbb-\aaa}{\kkk}\Big )^{\frac{1}{n+1}}\leq \frac{\ddd_{n+1}^{2}}{2^{n+1}}.$$
By \eqref{*V6*1} we obtain
\begin{equation*}
F_{n+1}(\bbb_{l})-F_{n+1}(\aaa_{l})\geq (\bbb_{l}-\aaa_{l})^{1+\frac{1}{n+1}}.
\end{equation*} 

\medskip

By construction, if $x\not\in H_{n+1}$, then with the above notation, $x$ belongs to an interval $J_{x}= [\aaa^x_{ l_{x}},\bbb^x_{ l_{x}}]$,  for  some interval $[\alpha^x,\beta^x]$ and some suitable integer $l_{x}$. 
We put $a_{x,n+1}=\aaa^x_{ l_{x}}$, $b_{x,n+1}=\bbb^x_{ l_{x}}$
and this choice yields part (iv)(c) of the induction.

\medskip

Finally, it remains to impose the behavior of $F_{n+1}$ on $H_{n+1}$. We impose  that  for every $i=1,...,n+1$, the function $F_{n+1}$
coincides with a Cantor type interpolating function on $[s_{i}+\ddd_{n+1}/2,
s_{i}+\ddd_{n+1}]$, is constant on $[s_{i}-\ddd_{n+1},s_{i}+\ddd_{n+1}/2]$
and $F_{n+1}(s_{i}+\ddd_{n+1})-F_{n+1}(s_{i}+\ddd_{n+1}/2)=F_{n+1}(s_{i}+\ddd_{n+1})-F_{n+1}(s_{i}-\ddd_{n+1})=\ddd_{n+1}^{2}.$  
Hence, the incerement of $F_{n+1}$ is defined on all components of $H_{n+1}$, and thus on $\zu$.

\medskip

Since $H_{n+1}$ consists of $n+1$ many component intervals,   \eqref{*V6*1} gives
\begin{equation}\label{*V8*3}
F_{n+1}(1)-F_{n+1}(0)\leq (n+1)\ddd_{n+1}^{2}+\frac{\ddd_{n+1}^{2}}{2^{n+1}}.
\end{equation}

\medskip

The attentive reader has checked that all the parts of the induction are verified for $n+1$ instead of $n$.

\medskip

\begin{definition}
Iterating the procedure, we construct a sequence of functions $(F_n)_{n\geq 1}$ and define the continuous strictly increasing function
$$F=\sum_{n=1}^{+\infty} F_{n}.$$ 
\end{definition}
The continuity follows from the uniform convergence guaranteed by \eqref{*V8*3}, and the strict monotonicity from  the density of the $I_n$ and the fact that the   $F_n$ do not increase at the same locations.

\subsection{Multifractal properties of $F$}

\begin{proposition}
For every  $x \in \zu$ which is not one of the $s_n$, one has $h_F(x)\leq 1$.
\end{proposition}
\begin{proof}
  Consider a real number  $x \in \zu$ which is not one of the $s_n$.

\begin{lemma}
\label{lemnotin}
There exists an infinite number of integers $n$ such that   $x\not\in H_{n }$.
\end{lemma}
\begin{proof}
Assume that there exists $n_{0}$ such that
$x\in H_{n_{0}} $. 
 Choose $n_{1}\geq n_{0}$
such that for $n_{0}\leq n <n_{1}$,
\begin{equation}\label{*V9*1}
x\in \bigcup_{i=1}^{n_{0}}[s_{i}-\ddd_{n},s_{i}+\ddd_{n}]\sse H_{n} \ \text{ but } \ x\not\in \bigcup_{i=1}^{n_{0}}[s_{i}-\ddd_{n_{1}},s_{i}+\ddd_{n_{1}}].
\end{equation}
This integer $n_1$ exists, because the sequence $\delta_n$ converges to zero. 

Since $x\in H_{n_{1}-1}$,    $x\not \in \widetilde I_{n_{1}}\sm
H_{n_{1}-1}.$
Moreover, \eqref{*V2*2}, \eqref{*sn1val*} and \eqref{*V9*1} imply that
$x\not \in \cup_{i=n_{0}+1}^{n_{1}-1}[s_{i}-\ddd_{n_{1}-1},
s_{i}+\ddd_{n_{1}-1}].$
By \eqref{*V5*2}, $x\not\in [s_{n_{1}}-\ddd_{n_{1}},s_{n_{1}}+\ddd_{n_{1}}]$, which is included in the interior of $ \widetilde I_{n_{1}}\sm H_{n_{1}-1}.$
Therefore $x\not \in H_{n_{1}}$. 

Since this argument can be repeated, 
there are infinitely many $n_{j}$'s, $j=1,2,...$
such that $x\not \in H_{n_{j}}$.
\end{proof}

Now, if $x\not \in H_{n }$, then item (iv)(c) of the induction provides us with an interval   $I_{x,n }=[a_{x,n },b_{x,n }]$
 such that  \eqref{*V3*2} holds with $n$. Since by Lemma \ref{lemnotin}, this holds for infinitely many intervals whose size goes to zero, 
this implies $h_{F}(x)=h_{\mmm}(x)\leq 1.$ \end{proof}
 
 We take care of the points $s_n$, where $F$ is more regular.

\begin{proposition}
For every $n\geq 1$, $h_F(s_n) =2$. Hence, $d_F(2) = 0$.
\end{proposition}

\begin{proof}
Consider one of the  points $ s_{n_{0}}$,  and $n\geq n_{0}$.  Then $s_{n_{0}}\in H_{n}$ and by part (iv)(d) of the induction, 
the functions $F_{p}$ with $1\leq p   <n$ are constant on
$[s_{n_{0}}-\ddd_{n},s_{n_{0}}+\ddd_{n}]$.

By  parts (iv)(a) and  (iv)(b) of the induction, we have 
\begin{equation}\label{*V8*1}
F_{n}(s_{n_{0}}+\ddd_{n})-F_{n}(s_{n_{0}})=
F_{n}(s_{n_{0}}+\ddd_{n})-F_{n}(s_{n_{0}}-\ddd_{n})=
\ddd_{n}^{2}.
\end{equation}

Suppose $y\in [s_{n_{0}}-\ddd_{n},s_{n_{0}}+\ddd_{n}]\sm
[s_{n_{0}}-\ddd_{n+1},s_{n_{0}}+\ddd_{n+1}]$.
Then 
\begin{eqnarray*}
&&|F(y)-F(s_{n_{0}})|\leq\\
&& \sum_{k=1}^{n-1}|F_{k}(y)-F_{k}(s_{n_{0}})|+
|F_{n}(y)-F_{n}(s_{n_{0}})|+\sum_{k=n+1}^{\oo}
|F_{k}(y)-F_{k}(s_{n_{0}})|\\
&& =0+\Delta_{n}+\Sigma_{n+1}.
\end{eqnarray*}

If $\ddd_{n}/2 \leq |y-s_{n_{0}}|\leq \delta_n$, then $|\Delta_{n}|\leq \ddd_{n}^{2}\leq 
4|y-s_{n_{0}}|^{2}$ and $|\Sigma_{n+1}|\leq \ddd_{n}^{2}\leq 
4|y-s_{n_{0}}|^{2}$.

If $\ddd_{n+1}<|y-s_{n_{0}}|< \ddd_{n}/2$ then $\Delta_{n}=0$. By using the definition of $H_{n+1}$
and the choice of $\ddd_{n}$ and $\ddd_{n+1}$, we obtain that
$|\Sigma_{n+1}|< 3\ddd_{n+1}^{2}+\sum_{k=n+2}^{\oo} (k\ddd_{k}^{2}+\frac{\ddd_{k}^{2}}{2^{k}})< 4|y-s_{n_{0}}|^{2}$.

We have used that  \eqref{*V8*2} and  \eqref{*V8*3bis} holds for all $n\geq n_{0}$. This combined with  \eqref{*V8*1} implies that $h_{F}(x)=h_{\mmm}(x)=h_{\mmm}(s_{n_{0}})=2$.  Since all the other points have an exponent less than 1, this concludes the proof.
\end{proof}

 \begin{proposition}
 For every $h\in \zu$, $d_F(h)=h$.
 \end{proposition}
 \begin{proof}
 Obviously, by Proposition \ref{prop2}, only the lower bound needs to be proved.
 
By construction, in each interval $J_l = [\alpha_l,\beta_l]$ (recall \eqref{defjl}),  if   $\widetilde F_{l}$  stands for the restriction
of $F$ onto $J_l$, one has   $d_{\widetilde F_{l}}(h)  \geq h$ for all $h\in [0,1]$.  Hence  $d_{\widetilde F_{l}} (h) =h$ for $h\in [0,1]$. The fact that $F$ is homogeneously multifractal follows from the density of the intervals  $I_n$ (which implies the density of the $J_l$).
 \end{proof}

%%%%%%%%%%%%%%%%%%%%%%%%%%%%%%%%%%
%%%%%%%%%%%%%%%%%%%%%%%%%%%%%%%%%%
%%%%%%%%%%%%%%%%%%%%%%%%%%%%%%%%%%
%%%%%%%%%%%%%%%%%%%%%%%%%%%%%%%%%%
%%%%%%%%%%%%%%%%%%%%%%%%%%%%%%%%%%
%%%%%%%%%%%%%%%%%%%%%%%%%%%%%%%%%%
%%%%%%%%%%%%%%%%%%%%%%%%%%%%%%%%%%

\section{Construction of a monotone function with H\"older exponents 1}\label{sechomone}
 
In this section we prove Proposition \ref{homone}. The function $Z$ we obtain is a sort of  monotone increasing Weierstrass-like function.

\begin{proof}

First we need  positive sequences $(a_{n})_{n\geq N}$ and $(b_{n})_{n\geq N}$.
The sequence $(a_{n})$ will tend to zero sufficiently fast,
while $(b_{n})$ will consist of integers  growing  to infinity.
The sequences are built inductively. Set $a_{1}=1$, $b_{1}=4$.
Suppose that $n >1$ and the terms of $a_{n'}$ and $b_{n'}$ have already been
selected for $n'<n $.
We select $a_{n }\in (0,1)$ and 
\begin{eqnarray}\label{*mod3}
\text{$b_{n}\in \N$ is  congruent to $3$ modulo $4$} \ \ \ \ \ \ \ \ \ \ \ \ \\
\mbox{such that } \ \ \ \ \ 
\label{*12070901a}
 \frac{a_{n-1}}{300}  >  a_{n} > (4b_{n })^{\frac{-1}{n }}  \  \ \mbox { and } \ \ \frac{a_{n}b_{n}}{100}
>
\sum_{n'=1}^{n-1}a_{n'}b_{n'}.
\end{eqnarray}
 Iterating the scheme gives the sequences $(a_{n})_{n\geq N}$ and $(b_{n})_{n\geq N}$.
From the left handside inequality in 
\eqref{*12070901a}, it follows that
\begin{equation}\label{*12070901b}
\frac{a_{n }}{100}
>
\sum_{n'=n +1}^{\oo}a_{n'}.
\end{equation}

Put $g(x)=0$ on $[0,1],$ $g(x)=1$ on $[2,3],$  
 $g(4)=0$ and suppose that $g$ is linear on
 $[1,2]$ and $[3,4]$, moreover to define $g$ on $\R$
we also assume that it is periodic by $4$ (see Figure \ref{fig5}).
Further, we set
\begin{equation*}
G(x)=\sum_{n=1}^{\oo}a_{n}g(b_{n}x)\text{ and }Z(x)=\int_{0}^{x}G(t)dt.
\end{equation*}

Then one can easily see that $G$ is continuous as uniform limit of continuous functions, and $Z$ is continuously differentiable
with $Z'(x)=G(x)$ for all $x\in [0,1].$
In particular,  we have $h_{Z}(x)\geq 1$ for all $x\in [0,1]$.
 If  we can verify that for every $\xo\in [0,1]$
and for every sufficiently large 
$n_{1}\in \N$ we can find $\xe$, depending on $n_{1}$,  such that 
$|x_{1}-x_{0}|\leq 4/b_{n_{1}}$ and
\begin{equation}\label{*12070903a}
|Z(\xe)-Z(\xo)-Z'(\xo)(\xe-\xo)|>\frac{1}{16^{2}}|\xe-\xo|^{1+\frac{1}{n_{1}}},
\end{equation}
then $h_{Z}(x)\leq 1$ for all $x\in [0,1]$. Combining this with the derivability of $Z$,   we obtain that $h_{Z}(x)=1$ for all $x\in [0,1]$,
proving Proposition \ref{homone}.

%%%%%%%%%%%%%%%%%%%%%%%%%%%%%%%%%% 
\begin{center}
\begin{figure}
  \includegraphics[width=8.0cm,height = 2.8cm]{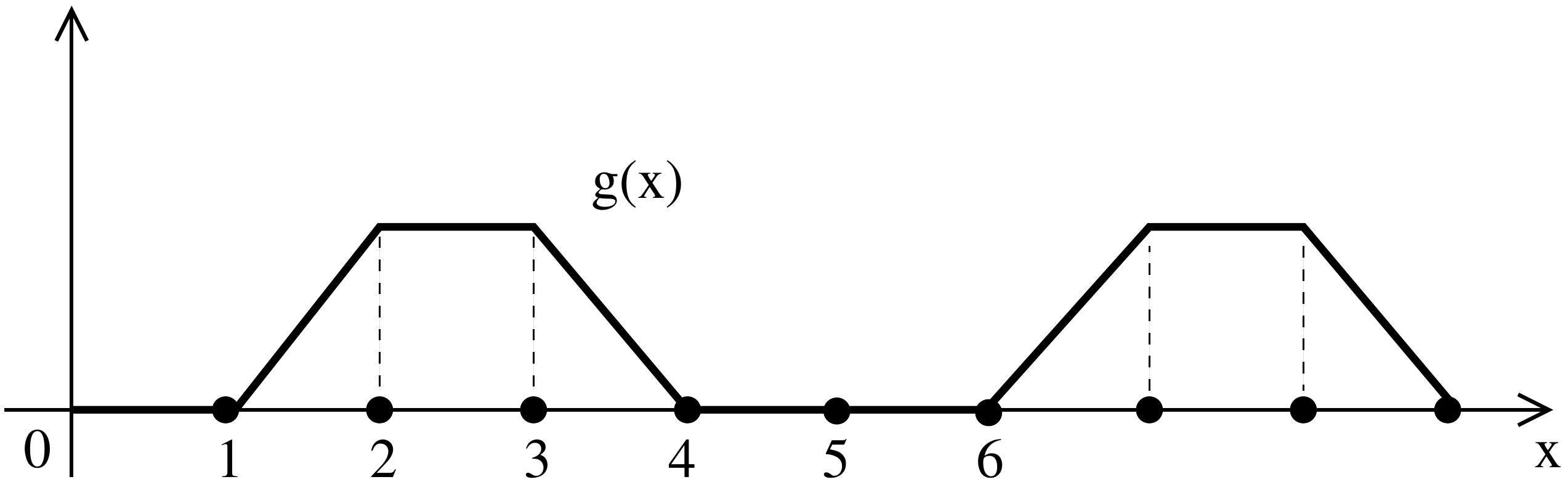}
\caption{The function $g$} \label{fig5}
\end{figure}
\end{center}
%%%%%%%%%%%%%%%%%%%%%%%%%%%%%%%%%% 

Observe that $g$ is a Lipschitz function with Lipschitz constant one, hence 
\begin{equation}\label{*12070903b}
|g(b_{n_{1}}x)-g(b_{n_{1}}x)|\leq b_{n_{1}}|x-y|\text{ for all }x,y\in\R.
\end{equation}

Fix $x_0\in \zu$. It remains us to check  that \eqref{*12070903a} holds true.
If $x_{0}=0$, then $g(b_{n_{1}}\xo)=0$
and if $x_{0}=1$, then $g(b_{n_{1}}\xo)=1$, by \eqref{*mod3}.

From the definition of $g$,   if $n_{1}$ is sufficiently large,  we can
find $\xe\in[0,1]$  such that:
\begin{itemize}
\smallskip\item
 $|\xe-\xo|\leq {4}/{b_{n_{1}}}$,
\smallskip\item
$g(b_{n_{1}}x)-g(b_{n_{1}}\xo)$ is of constant sign for $x$ in the interval  $I_{0,1}$ with endpoints
$\xo$ and $\xe$,
\smallskip\item  there exists a subinterval $I'_{0,1}\sse I_{0,1}$ of length
$\frac{1}{b_{n_{1}}}$  such that $|g(b_{n_{1}}x)-g(b_{n_{1}}\xo)|\geq 1/2$
for all $x\in I'_{0,1}$.
\end{itemize}
 
Without limiting generality, we suppose that  $\xe>\xo$, $[\xo,\xe]=I_{1,0}$, 
and 
\begin{eqnarray}\label{*12070902a}
|\xe-\xo|\leq  {4}/{b_{n_{1}}},  \  \  \ g(b_{n_{1}}x)-g(b_{n_{1}}\xo)\geq 0 
\text{ for }x\in  I_{1,0}\\
\nonumber
g(b_{n_{1}}x)-g(b_{n_{1}}\xo)\geq {1}/{2}  \ \ \text{ for all }
x\in [x',x'+ {1}/{b_{n_{1}}}]=I'_{0,1} \sse I_{1,0}.
\end{eqnarray}

Then
\begin{eqnarray*}
I & :=  & |Z(\xe)-Z(\xo)-Z'(\xo)(\xe-\xo)| = \Big |\int_{\xo}^{\xe}G(t)dt-G(\xo)(\xe-\xo)\Big |\\
& = & \Big |\sum_{n=1}^{\oo}a_{n}\Big (\int_{\xo}^{\xe}g(b_{n}t)dt-g(b_{n}\xo )(\xe-\xo)\Big )\Big |
 \end{eqnarray*}
by the uniform convergence of the series. Dividing the sum into three parts, and   using successively \eqref{*12070903b},  \eqref{*12070902a},  and the fact that $|g|\leq 1$, we obtain
\begin{eqnarray*} 
I & = & \Big |\sum_{n=1}^{n_{1}-1}\Big (a_{n}\int_{\xo}^{\xe}g(b_{n}t)-g(b_{n}\xo)dt\Big )
+a_{n_{1}}\int_{\xo}^{\xe}g(b_{n_{1}}t)-g(b_{n_{1}}\xo)dt\\
&&
\sum_{n=n_{1}+1}^{\oo}\Big (a_{n}\int_{\xo}^{\xe}g(b_{n}t)-g(b_{n}\xo)dt\Big ) \Big |\\
&\geq&
a_{n_{1}}\int_{\xo}^{\xe}g(b_{n_{1}}t)-g(b_{n_{1}}\xo)dt-
\sum_{n=1}^{n_{1}-1}\Big (a_{n}\int_{\xo}^{\xe}|g(b_{n}t)-g(b_{n}\xo)|dt\Big )
\\
&&-\sum_{n=n_{1}+1}^{\oo}\Big (a_{n}\int_{\xo}^{\xe}|g(b_{n}t)-g(b_{n}\xo|dt\Big )\\
&\geq&  \frac{a_{n_{1}}}{2b_{n_{1}}}- \sum_{n=1}^{n_{1}-1}\Big (a_{n}\int_{\xo}^{\xe}b_{n}|t-\xo|dt\Big )-\sum_{n=n_{1}+1}^{\oo}a_{n}|\xe-\xo| 
  \end{eqnarray*}
Then, by  \eqref{*12070901a}, \eqref{*12070901b} and \eqref{*12070902a}, one finally has
\begin{eqnarray*}
I &  \geq &  \frac{a_{n_{1}}}{2b_{n_{1}}}- \sum_{n=1}^{n_{1}-1}a_{n}b_{n}\frac{16 }{2b_{n_{1}}^{2}}-\sum_{n=n_{1}+1}^{\oo}a_{n}\frac4{b_{n_{1}}}\geq
\frac{a_{n_{1}}}{4b_{n_{1}}}\\
& \geq & (4b_{n_{1}})^{-1-\frac{1}{n_{1}}}\geq \Big (\frac{|\xe-\xo|}{16}\Big )^{1+\frac{1}{n_{1}}}.
\end{eqnarray*}
This implies \eqref{*12070903a}.
\end{proof} 

%%%%%%%%%%%%%%%%%%%%%%%%%%%%%%%%%%
%%%%%%%%%%%%%%%%%%%%%%%%%%%%%%%%%%
%%%%%%%%%%%%%%%%%%%%%%%%%%%%%%%%%%
%%%%%%%%%%%%%%%%%%%%%%%%%%%%%%%%%%
%%%%%%%%%%%%%%%%%%%%%%%%%%%%%%%%%%
%%%%%%%%%%%%%%%%%%%%%%%%%%%%%%%%%%
%%%%%%%%%%%%%%%%%%%%%%%%%%%%%%%%%%
%%%%%%%%%%%%%%%%%%%%%%%%%%%%%%%%%%
%%%%%%%%%%%%%%%%%%%%%%%%%%%%%%%%%%
%%%%%%%%%%%%%%%%%%%%%%%%%%%%%%%%%%
%%%%%%%%%%%%%%%%%%%%%%%%%%%%%%%%%%
%%%%%%%%%%%%%%%%%%%%%%%%%%%%%%%%%%
%%%%%%%%%%%%%%%%%%%%%%%%%%%%%%%%%%
%%%%%%%%%%%%%%%%%%%%%%%%%%%%%%%%%%
%%%%%%%%%%%%%%%%%%%%%%%%%%%%%%%%%%
%%%%%%%%%%%%%%%%%%%%%%%%%%%%%%%%%%
%%%%%%%%%%%%%%%%%%%%%%%%%%%%%%%%%%


\begin{thebibliography}{9}

 

\bibitem{BS1}{J. Barral, S. Seuret}, From multifractal measures to multifractal wavelet series,   {\it J. Fourier Anal. and App.,} 11(5), (2005), 589-614.

\bibitem{BMP} { G. Brown \and G. Michon \and
J. Peyri\`ere}, On the multifractal analysis of measures, \emph{J. Stat. Phys.} {\bf 66} (1992), 775--790.

 
\bibitem{DLVM} 
    K. Daoudi, J. L\'evy V\'ehel, Y. Meyer, 
    Construction of continuous functions with prescribed local regularity, 
    {\it Constr. Approx} 14 (3),   349-385 (1988).
    
 \bibitem{Fa1} {K. J. Falconer,} {\it Fractal Geometry,}
John Wiley \&
Sons, (1990).

 
\bibitem {F2} K. J. Falconer,  {\it Techniques in Fractal Geometry,} Wiley, New York (1997).
 
 
 \bibitem{Feng} D.-J. Feng, Multifractal analysis of Bernoulli convolutions associated with Salem numbers. {\it Adv. Math.} {\bf 229}, (2012), 3052--3077.
 
 \bibitem{J2}  { S. Jaffard},   Construction de fonctions multifractales ayant un spectre de singularit\'es prescrit. (French) [Construction of multifractal functions with a prescribed spectrum of singularities], {\it C. R. Acad. Sci. Paris} Ser. I Math. 315   no. 1, (1992), 19-24. 
 
\bibitem{JAFFPRES} S. Jaffard,  Construction of functions with prescribed H\"older and Chirps exponents, {\it  Rev. Mat. Iberoamericana.}  {\bf 16}(2), (2000),    331--349.

 
\bibitem{JAFFRIEMANN}  { S. Jaffard}, The spectrum of singularities of Riemann's function. {\it Rev. Mat. Iberoamericana,} 12(2),  (1996), 441--460. 

\bibitem{mattila} P. Mattila, {\it Geometry of {S}ets and {M}easures in
{E}uclidian {S}paces,} Cambridge Studies in Advanced Mathematics,
Cambridge University Press (1995).

\bibitem{MEYER} Y. Meyer, {\it Wavelets and operators,} Hermann, 1990.

\bibitem{MR} R.H. Riedi, B.B.  Mandelbrot, Inverse measures, the
inversion formula, and discontinuous multifractals, {\it Adv. Appl. Math.}
{\bf 18},  (1997), 50--58.

\bibitem{SJLV} S. Seuret, J. L\'evy V\'ehel, The local H\"older function of a continuous function,    {\it Appl. Comput. Harmon. Anal.}, 13(3), (2002), 263--276. 

 \end{thebibliography}
\end{document}